\documentclass[11pt]{amsart}
\usepackage{amsmath,amssymb,a4wide,scalerel}
\usepackage{mathabx}
\usepackage{color}
\usepackage{xcolor,graphicx}
\usepackage{mathrsfs}
\usepackage{tikz,pgfplots}
\usepackage{subcaption}
\usepackage[normalem]{ulem}
\usepackage{hyperref}
\usepackage{dsfont}
\usepackage{comment}
\usepackage{enumitem}

\newtheorem{theorem}{Theorem}
\newtheorem{lemma}[theorem]{Lemma}
\newtheorem{proposition}[theorem]{Proposition}
\newtheorem{remark}[theorem]{Remark}
\newtheorem{corollary}[theorem]{Corollary}

\newtheorem{assumption}{Assumption}
\newtheorem{example}{Example}

\newcommand{\dt}{\tau}

\newcommand{\R}{{\mathbb R}}
\newcommand{\N}{{\mathbb N}}
\newcommand{\E}{{\mathbb E}}

\newcommand{\e}{\mathrm{e}}

\newcommand*\diff{\mathop{}\!\mathrm{d}}

\usepackage{color}

\author{Charles-Edouard Br\'ehier}
              \address{Universite de Pau et des Pays de l'Adour, E2S UPPA, CNRS, LMAP,Pau, France}
              \email{charles-edouard.brehier@univ-pau.fr}

\author{David Cohen}
              \address{Department of Mathematical Sciences,
              Chalmers University of Technology and University of Gothenburg, 41296~Gothenburg, Sweden}
              \email{\tt david.cohen@chalmers.se}

\author{Yoshio Komori}
              \address{Department of Physics and Information Technology, Kyushu Institute of Technology,
              680-4 Kawazu, Iizuka, 820-8502, Japan}
              \email{\tt komori@phys.kyutech.ac.jp}

\begin{document}

\title[Conformal integrators for linearly damped stochastic Poisson systems]{Stochastic conformal integrators for linearly damped stochastic Poisson systems}

\begin{abstract}
We propose and study conformal integrators for linearly damped stochastic Poisson systems.
We analyse the qualitative and quantitative properties of these numerical integrators: preservation
of dynamics of certain Casimir and Hamiltonian functions,
almost sure bounds of the numerical solutions, and strong and weak rates of convergence under appropriate conditions.
These theoretical results are illustrated with several numerical experiments on, for example,
the linearly damped free rigid body with random inertia tensor or
the linearly damped stochastic Lotka--Volterra system.
\end{abstract}

\maketitle
{\small\noindent
{\bf AMS Classification.} 60H10, 60H35, 65C20, 65C30, 65P10.

\bigskip\noindent{\bf Keywords.} Stochastic differential equations. Linearly damped stochastic Poisson systems.
Casimir and Hamiltonian functions. Geometric numerical integration. Stochastic conformal integrator.
Strong and weak convergence.

\section{Introduction}
The design and analysis of structure-preserving numerical methods,
i.\,e. Geometric Numerical Integration (GNI), has been a major focus of research in
numerical analysis of Ordinary Differential Equations (ODEs) for $35$ years,
see for instance \cite{MR2839393, MR1270017, MR2840298, MR2132573, MR3642447}.
Prominent examples of applications of GNI are Hamiltonian and Poisson systems
of classical mechanics. When such systems are subject to dissipation,
due to a linear (possibly time-dependent) damping term, one speaks of conformal Hamiltonian and Poisson
systems. Efficient numerical methods for conformal ordinary and partial differential equations have been proposed and studied in e.g.
\cite{MR2573261,MR2855436,MR2994297,MR3456972,MR3516867,MR4242161,MR4300344,MR4486639} and references therein.

The focus of the this work is on the design and analysis of conformal exponential integrators for
randomly perturbed linearly damped Poisson systems of the form
\begin{equation}\label{eq:SDEgeneralintro}
\diff y(t)=\Bigl(B(y(t))\nabla H_0(y(t))-\gamma(t)y(t)\Bigr)\diff t+\sum_{m=1}^M B(y(t))\nabla H_m(y(t))\circ\,\diff W_m(t)
\end{equation}
with a structure matrix $B\colon\mathbb R^d\to\mathbb R^{d\times d}$, Hamiltonian functions $H_j\colon\mathbb R^d\to\mathbb R$,
for $j=0,1,\ldots,M$, a damping term $\gamma\colon\mathbb R\to\mathbb R$, and independent standard real-valued Wiener process $W_m$ for
$m=1,\ldots,M$. This Stochastic Differential Equation (SDE) is understood in the Stratonovich sense, which as usual is indicated by the symbol $\circ$ in the SDE. See Section~\ref{sec-sett} for the precise setting and details on the notation.  Examples of systems which belong to the general class of systems~\eqref{eq:SDEgeneralintro} above are given by
\begin{equation}\label{eq:SDEspecificintro}
\diff y(t)=B(y(t))\nabla H(y(t))\bigl(\diff t+c\circ\,\diff W(t)\bigr)-\gamma(t)y(t)\diff t,
\end{equation}
where $c>0$. In the above problem, one takes $M=1$, $W=W_1$, $H_0=H$ and $H_1=cH$ in the SDE~\eqref{eq:SDEgeneralintro}.
Some results which are specific to this subclass of systems will be given in this work.

Our work is built upon recent developments for linearly damped stochastic Hamiltonian systems in \cite{MR3898821,MR4512617}
and for undamped stochastic Lie--Poisson systems \cite{MR4593213}, as well as on early works in the deterministic case from \cite{MR2784654,MR4242161}.

The contributions of our work are the following.
\begin{itemize}
\item We design stochastic exponential integrators for~\eqref{eq:SDEgeneralintro} and we show that quadratic Casimirs are damped accordingly with the evolution law of the exact solution, i.\,e. that the proposed scheme is a stochastic conformal integrator, see Proposition~\ref{prop:confCasimir}. Moreover, for the specific class of systems~\eqref{eq:SDEspecificintro}, we show that if the Hamiltonian function is homogeneous of degree $p$ then its damping behavior is also preserved by the proposed integrator, see Proposition~\ref{prop:balanceNum}.
\item We show that under certain conditions on Casimir and Hamiltonian functions, the results above provide almost sure bounds for the exact and numerical solutions, see Corollaries~\ref{cor:exist} and~\ref{cor:num}. This is a crucial step when the drift and diffusion coefficients of the SDEs~\eqref{eq:SDEgeneralintro}~and~\eqref{eq:SDEspecificintro}
are not assumed to be globally Lipschitz continuous, which happens in some of the considered examples.
\item We prove strong and weak convergence results for the proposed stochastic exponential integrators, under the conditions ensuring almost sure moment bounds on the exact and numerical solutions. We show that in general the strong and weak rates of convergence are equal to $1/2$ and $1$ respectively, see Theorems~\ref{th:strong} and~\ref{th:weak}. Moreover, we show that when $M=1$, the strong rate of convergence is equal to $1$, see Theorems~\ref{th:strongM1} and~\ref{th:strong1}.
\item We provide numerical experiments in order to illustrate the qualitative behavior and the convergence results for the proposed stochastic conformal integrators.
\end{itemize}

The paper is organized as follows. Section~\ref{sec-sett} presents the setting, the main examples and the main qualitative properties
of linearly damped stochastic Poisson systems considered in this article. We then introduce and analyse the qualitative properties of the proposed stochastic conformal exponential integrators in Section~\ref{sec-scheme}. In Section~\ref{sec:convergence}, we state and prove the strong and weak convergence results for our numerical schemes.
Finally, qualitative and quantitative properties of the studied numerical methods are illustrated in several
numerical experiments in Section~\ref{sec-numexp}.

\section{Setting}\label{sec-sett}

In this section, we first provide notation and define the class of SDEs considered in this article.
Next, we describe examples which fit in this class of linearly damped stochastic Poisson systems.
Finally, we study the main qualitative properties of their solutions.

\subsection{Notation}

Let $d$ be a positive integer which denotes the dimension of the considered problem. The Poisson structure is determined by a mapping $B\colon\mathbb{R}^d\to\mathbb{R}^{d\times d}$ which takes values in the set of skew--symmetric matrices, i.\,e. one has $B(y)^T=-B(y)$ for all $y\in\mathbb{R}^d$.

Let the positive integer $M$ denote the dimension of the stochastic perturbation, and consider  independent standard real-valued Wiener processes $W_1=\bigl(W_1(t)\bigr)_{t\ge 0},\ldots,W_M=\bigl(W_M(t)\bigr)_{t\ge 0}$ defined on a probability space denoted by $(\Omega,\mathcal{F},\mathbb{P})$. In addition, consider Hamiltonian functions
$H_0, H_1, \ldots, H_M\colon\mathbb{R}^d\to\mathbb{R}$. Finally, let $\gamma:\mathbb{R}\to\mathbb{R}$ denote the damping function, which is assumed to be continuous.

In this article, we consider two classes of linearly damped stochastic Poisson systems. The first class of systems is defined as
\begin{equation}\label{prob}
\left\lbrace
\begin{aligned}
\text dy(t)&=\Bigl(B(y(t))\nabla H_0(y(t))-\gamma(t)y(t)\Bigr)\,\text dt+\sum_{m=1}^M B(y(t))\nabla H_m(y(t))\circ\,\text dW_m(t),\quad t\ge 0,\\
y(0)&=y_0,
\end{aligned}
\right.
\end{equation}
where $y_0$ is a given initial value (which is assumed to be deterministic). The noise in the SDE~\eqref{prob} is understood in the Stratonovich sense.

Assume that the structure matrix $B$ is a mapping of class $\mathcal{C}^2$, that the Hamiltonian function $H_0$ is of class $\mathcal{C}^2$, and that the Hamiltonian functions $H_1,\ldots,H_M$ are of class $\mathcal{C}^3$. Under those regularity conditions, the SDE~\eqref{prob} admits a unique local  solution (defined for $t\in[0,\mathbf{T})$, where $\mathbf{T}$ is a stopping time with values in $[0,\infty]$), see for instance \cite[Chapter 2]{MR2380366} or \cite[Section~4.8]{MR1214374}.

We also study a second class of linearly damped stochastic Poisson systems, given by
\begin{equation}\label{prob1}
\left\lbrace
\begin{aligned}
\diff y(t)&=B(y(t))\Bigl(\nabla H(y(t))\diff t+c\nabla H(y(t))\circ\diff W(t)\Bigr)-\gamma(t)y(t)\diff t,\quad t\ge 0\\
y(0)&=y_0,
\end{aligned}
\right.
\end{equation}
where $c>0$ measures the size of the random perturbation. The system of SDEs~\eqref{prob1} can be obtained from the first class~\eqref{prob} with a single Wiener process, i.\,e. $M=1$, and with a single Hamiltonian function $H_0=H$ and $H_1=cH$. This class of SDEs is a damped version of randomised Hamiltonian/Poisson systems, see for instance \cite[Chap. V.4]{MR629977}, \cite{MR1735312}, \cite[Sect. 3.1]{MR2408499}, or \cite{MR3210739}. Local well-posedness of~\eqref{prob1} is ensured by assuming that
the structure matrix $B$ is a mapping of class $\mathcal{C}^2$ and that the Hamiltonian function $H$ is of class $\mathcal{C}^3$,
the solution is then defined for $t\in[0,\mathbf{T})$, where $\mathbf{T}$ is a stopping time with values in $[0,\infty]$.

Global well-posedness for~\eqref{prob} and~\eqref{prob1}, i.\,e. having $\mathbf{T}=\infty$ almost surely, requires additional assumptions.
A sufficient condition which exploits the structure of the systems~\eqref{prob} and~\eqref{prob1} will be stated below.




\subsection{Examples}

We now give some examples of linearly damped stochastic Poisson systems. All these examples are considered in the numerical experiments in Section~\ref{sec-numexp}.

\begin{example}[Linearly damped stochastic Hamiltonian systems]\label{exppend}
Let $d=2n$ be an even integer and decompose $y=(y_1,y_2)\in\mathbb{R}^{n}\times\mathbb{R}^n$. Let $I\in\mathbb{R}^{n\times n}$ be the identity matrix in dimension $n$ and consider the symplectic matrix $J=\begin{pmatrix} \phantom{+}0 & I \\ -I & 0\end{pmatrix}\in\mathbb{R}^{d\times d}$. Choosing $B(y)=J^{-1}=\begin{pmatrix} 0 & -I \\ I & \phantom{+}0\end{pmatrix}$ for all $y\in\mathbb{R}^d$ for the structure matrix, the system~\eqref{prob} gives
the linearly damped stochastic Hamiltonian system
$$
\diff
\begin{pmatrix}
y_1(t)\\
y_2(t)
\end{pmatrix}
=
\begin{pmatrix}
-\nabla_{y_2}H_0(y_1(t),y_2(t))\\
\phantom{+}\nabla_{y_1}H_0(y_1(t),y_2(t))
\end{pmatrix}
\diff t
-\gamma(t)
\begin{pmatrix}
y_1(t)\\
y_2(t)
\end{pmatrix}
\diff t
+\sum_{m=1}^{M}
\begin{pmatrix}
-\nabla_{y_2}H_m(y_1(t),y_2(t))\\
\phantom{+}\nabla_{y_1}H_m(y_1(t),y_2(t))
\end{pmatrix}
\circ\,\diff W_m(t)
$$
This class of problems has been studied in \cite{MR4512617,MR3898821} for instance.

Considering the second class of systems~\eqref{prob1}, choosing $H(y)=\frac12y_1^2-\cos(y_2)$, one obtains a linearly damped version of the stochastic mathematical pendulum studied for instance in \cite{MR3210739}:
\begin{equation}\label{pendulum}
\diff
\begin{pmatrix}
y_1(t)\\
y_2(t)
\end{pmatrix}
=
\begin{pmatrix}
-\sin(y_2(t))\\
y_1(t)
\end{pmatrix}
\left(\diff t+c\circ\diff W(t)\right)-\gamma(t)y(t)\diff t.
\end{equation}
Note that the drift and diffusion coefficients in~\eqref{pendulum} have bounded first and second order derivatives. It is thus straightforward to check that this system of SDEs is globally well-posed.
\end{example}

Let us now describe several examples with non-constant structure matrices $B(y)$. These provide linearly damped versions of the stochastic Poisson systems studied in \cite{MR3210739,MR4203112,MR4396382,MR4593213,MR4684229,ephrati2024exponentialmapfreeimplicit,dambrosio2024strongbackwarderroranalysis}.

\begin{example}[Linearly damped stochastic rigid body system]\label{exprigid}
Let $d=3$, denote $y=(y_1,y_2,y_3)\in\mathbb{R}^3$, and let $M=3$. Consider the structure matrix
\[
B(y)=
\begin{pmatrix}
0 & -y_3 & y_2\\
y_3 & 0 & -y_1\\
-y_2 & y_1 & 0
\end{pmatrix}
,
\quad \forall y\in\mathbb{R}^3.
\]
Given two families $(I_1,I_2,I_3)$ and $(\widehat{I}_1,\widehat{I}_2,\widehat{I}_3)$ of pairwise distinct positive real numbers (called principal moments of inertia), defined the Hamiltonian functions $H_0(y)=\frac12\left(\frac{y_1^2}{I_1}+\frac{y_2^2}{I_2}+\frac{y_3^2}{I_3}\right)$ and $H_m(y)=\frac12\frac{y_m^2}{\widehat I_m}$ for $m=1,2,3$ and for all $y\in\mathbb{R}^3$. One then obtains a damped version of the stochastic rigid body system given in \cite[Example~2.4]{MR4593213} when considering the system~\eqref{prob}:
\begin{align}\label{srbintro}
\diff
\begin{pmatrix}
y_1(t)\\
y_2(t)\\
y_3(t)
\end{pmatrix}
&=
\begin{pmatrix}
(I_3^{-1}-I_2^{-1})y_3(t)y_2(t)\\
(I_1^{-1}-I_3^{-1})y_1(t)y_3(t)\\
(I_2^{-1}-I_1^{-1})y_2(t)y_1(t)
\end{pmatrix}
\diff t
-\gamma(t)
\begin{pmatrix}
y_1(t)\\
y_2(t)\\
y_3(t)
\end{pmatrix}
\diff t
+
\begin{pmatrix}
0\\
y_1(t)y_3(t)/\widehat{I}_1\\
y_2(t)y_1(t)/\widehat{I}_1
\end{pmatrix}
\circ\diff W_1(t)\\
&+
\begin{pmatrix}
-y_3(t)y_2(t)/\widehat{I}_2\\
0\nonumber\\
y_1(t)y_2(t)/\widehat{I}_2
\end{pmatrix}
\circ\diff W_2(t)
+
\begin{pmatrix}
y_2(t)y_3(t)/\widehat{I}_3\\
-y_1(t)y_3(t)/\widehat{I}_3\\
0
\end{pmatrix}
\circ\diff W_3(t).
\end{align}
\end{example}

\begin{example}[Linearly damped stochastic Lotka--Volterra system]\label{explotka}
Let $d=3$ and denote $y=(y_1,y_2,y_3)\in\mathbb{R}^3$. Let $a,b$ be two real numbers. Consider the structure matrix
\[
B(y)=y_1^{1-ab}y_2^{b+1}\begin{pmatrix}0 & a & 1\\ -a & 0 & -ab\\-1 & ab & 0\end{pmatrix},\quad \forall~y\in\mathbb{R}^3,
\]
and the Hamiltonian function
\[
H(y)=y_1^{ab}y_2^{-b}y_3,\quad \forall~y\in\mathbb{R}^3.
\]
One then obtains a stochastic version of the linearly damped Lotka--Volterra system from~\cite{MR4242161} when considering the system~\eqref{prob1}:
\begin{align}\label{lotkaintro}
\diff
\begin{pmatrix}
y_1(t)\\
y_2(t)\\
y_3(t)
\end{pmatrix}
=
\begin{pmatrix}
y_1(t)(y_2(t)-aby_3(t))\\
y_2(t)(-a^2by_3(t)-aby_1(t))\\
y_3(t)(-aby_2(t)-ab^2y_1(t))
\end{pmatrix}
\diff t
+c
\begin{pmatrix}
y_1(t)(y_2(t)-aby_3(t))\\
y_2(t)(-a^2by_3(t)-aby_1(t))\\
y_3(t)(-aby_2(t)-ab^2y_1(t))
\end{pmatrix}
\circ\diff W(t)
-\gamma(t)
\begin{pmatrix}
y_1(t)\\
y_2(t)\\
y_3(t)
\end{pmatrix}
\diff t.
\end{align}

\end{example}

\begin{example}[Linearly damped stochastic Maxwell--Bloch system]\label{expmax}
Let $d=3$, denote $y=(y_1,y_2,y_3)\in\mathbb{R}^3$ and $M=1$. Consider the structure matrix
\[
B(y)=\begin{pmatrix}0 & -y_3 & y_2\\ y_3 & 0 & 0\\-y_2 & 0 & 0\end{pmatrix},\quad\forall~y\in\mathbb{R}^3,
\]
and the Hamiltonian functions
$$
H_0(y)=\frac12y_1^{2}+y_3\quad \text{and}\quad H_1(y)=y_3,\quad\forall~y\in\mathbb{R}^3.
$$
One then obtains a damped version of the stochastic Maxwell--Bloch system given in \cite[Example~2.3]{MR4593213} when considering the system~\eqref{prob}:
\begin{align}\label{maxintro}
\diff
\begin{pmatrix}
y_1(t)\\
y_2(t)\\
y_3(t)
\end{pmatrix}
=
\begin{pmatrix}
y_2(t)y_3(t)\\
y_1(t)y_3(t)\\
-y_1(t)y_2(t)
\end{pmatrix}
\diff t
-\gamma(t)
\begin{pmatrix}
y_1(t)\\
y_2(t)\\
y_3(t)
\end{pmatrix}
\diff t
+
\begin{pmatrix}
y_2(t)\\
0\\
0
\end{pmatrix}
\circ\diff W(t).
\end{align}

\end{example}

Finally, we also consider the following stochastic Poisson system in dimension $3$ with a constant structure matrix inspired by \cite{MR4396382}.

\begin{example}[Linearly damped stochastic three-dimensional Poisson system]\label{expcao}
Let $d=3$ and denote $y=(y_1,y_2,y_3)\in\mathbb{R}^3$. Consider the structure matrix
\[
B(y)=\begin{pmatrix}0 & 1 & -1\\ -1 & 0 & 1\\1 & -1 & 0\end{pmatrix},\quad \forall~y\in\mathbb{R}^3,
\]
and the Hamiltonian function
\[
H(y)=\sin(y_1)+\sin(y_2)+\sin(y_3),\quad \forall~y\in\mathbb{R}^3.
\]
One then obtains a damped version of the stochastic Poisson system from \cite{MR4396382} when considering the system~\eqref{prob1}:
\begin{align}\label{caointro}
\diff
\begin{pmatrix}
y_1(t)\\
y_2(t)\\
y_3(t)
\end{pmatrix}
=
\begin{pmatrix}
\cos(y_2(t))-\cos(y_3(t))\\
-\cos(y_1(t))+\cos(y_3(t))\\
\cos(y_1(t))-\cos(y_3(t))
\end{pmatrix}
\diff t
+c
\begin{pmatrix}
\cos(y_2(t))-\cos(y_3(t))\\
-\cos(y_1(t))+\cos(y_3(t))\\
\cos(y_1(t))-\cos(y_3(t))
\end{pmatrix}
\circ\diff W(t)
-\gamma(t)
\begin{pmatrix}
y_1(t)\\
y_2(t)\\
y_3(t)
\end{pmatrix}
\diff t.
\end{align}
Note that the drift and diffusion coefficients in this example have bounded first and second order derivatives. It is thus straightforward to check that this system of SDEs is globally well-posed.

\end{example}

\begin{remark}
In this article, we do not impose that the structure matrix $B(y)$ satisfies the Jacobi identity, and we thus do not investigate the property of the SDE~\eqref{prob} being conformal Poisson symplectic. We refer to~\cite[Sect. 3]{MR3898821} for results in the stochastic Hamiltonian setting, when $B(y)=J^{-1}$ with the standard symplectic matrix $J$, like in Example~\ref{exppend}. The Jacobi identity is verified for Examples~\ref{exprigid} and~\ref{expmax}, which are damped version of the stochastic Lie--Poisson systems studied for instance in~\cite{MR4593213}. The Jacobi identity is also satisfied in Examples~\ref{exppend} and~\ref{expcao}. We mention that even in the deterministic setting conformal Poisson symplectic numerical integrators are not constructed for general classes of problems, but only for some specific examples, see~\cite[Section~VII.4.2]{MR2840298}.
\end{remark}


\subsection{Qualitative properties}

In this section, we study the qualitative behavior of the solutions of the linearly damped stochastic Poisson systems~\eqref{prob} and~\eqref{prob1}. Let us recall that a mapping $C\colon\mathbb{R}^d\to\mathbb{R}$ (of class $\mathcal{C}^1$) is called a Casimir function of undamped deterministic or stochastic versions of~\eqref{prob} ($\gamma\equiv0$) if
\[
\nabla C(y)^TB(y)=0,\quad \forall~y\in\mathbb{R}^d,
\]
see for instance \cite{MR2840298,MR4593213} and references therein. For damped systems, a Casimir function is called a conformal Casimir function.


Morever, let us recall that a function $F\colon\mathbb R^d\to \mathbb R$ is called homogeneous of degree $p$ if one has
\begin{equation*}\label{homog1}
F(\lambda y)=\lambda^p F(y),\quad \forall~\lambda\in(0,\infty),~\forall~y\in\R^d.
\end{equation*}
If the homogeneous function $F$ is of class $\mathcal{C}^1$, the property above yields the identity
\begin{equation*}
\nabla F(y)^Ty=pF(y),\quad \forall~y\in\R^d.
\end{equation*}
For a continuous homogeneous function $F$ of degree $p$, note that one has $F(0)=0$. In addition, if one sets ${\bf m}(F)=\underset{y\in\R^d;~\|y\|=1}\min~F(y)$,
then for all $y\neq 0$ one has
\begin{equation*}\label{homog3}
F(y)=\|y\|^p F\left(\frac{y}{\|y\|}\right)\ge {\bf m}(F)\|y\|^p.
\end{equation*}
The class of homogeneous functions covers the cases of linear functionals $F(y)=\beta^Ty$, where $\beta\in\R^d$ ($p=1$) and of quadratic functionals $F(y)=\frac12y^TDy$, where $D$ is a symmetric matrix ($p=2$), and more generally of homogeneous polynomials of degree $p$. This also covers the case $F(y)=\zeta\prod_{k=1}^dy_k^{p_k}$ with $\zeta\in\mathbb R$ and $p_k\in\mathbb R$ for $k=1,\ldots,d$, see for instance~\cite{MR4242161} in the deterministic context.

\begin{remark}\label{rem:CasH}
Let us discuss the existence and properties of Casimir and Hamiltonian functions for the examples described in Section~\ref{sec-sett}.
\begin{itemize}
\item The linearly damped stochastic pendulum problem from Example~\ref{exppend} does not have a Casimir function.
The Hamiltonian $H(y)=\frac{y_1^2}2-\cos(y_2)$ is not a homogeneous function.
\item The linearly damped stochastic free rigid body from Example~\ref{exprigid} has
the quadratic Casimir function $C(y)=\frac12\left(y_1^2+y_3^2+y_3^2\right)$, which satisfies ${\bf m}(C)=\frac12$. The Hamiltonian functions $H_0$ and $H_m$, for $m=1,2,3$, are homogeneous functions of degree $2$, which satisfy ${\bf m}(H_0)>0$ and ${\bf m}(H_1)={\bf m}(H_2)={\bf m}(H_3)=0$.
\item For the stochastic Lotka--Volterra problem from Example~\ref{explotka}, the Hamiltonian function $H(y)=y_1^{ab}y_2^{-b}y_3$ is homogeneous of degree $ab-b+1$,
which satisfies ${\bf m}(H)=0$. This system has the Casimir function $C(y)=ab\ln(y_1)-b\ln(y_2)+\ln(y_3)$ which is not quadratic.
\item The linearly damped stochastic Maxwell--Bloch system from Example~\ref{expmax} has the quadratic Casimir $C(y)=\frac12(y_2^2+y_3^2)$, which satisfies ${\bf m}(C)=0$.
The Hamiltonian function $H_0(y)=\frac12y_1^2+y_2$ is not homogeneous.
\item The system of SDEs from Example~\ref{expcao} has the quadratic Casimir $C(y)=\frac12y^TDy$ with the matrix
$$
\begin{pmatrix}
1 & 1 & 1 \\
1 & 1 & 1 \\
1 & 1 & 1
\end{pmatrix},
$$
which satisfies ${\bf m}(C)=0$ (since the matrix $D$ is not invertible).
The Hamiltonian of this system is not an homogeneous function.
\end{itemize}
\end{remark}

Let us recall that for the undamped version of the stochastic Poisson system~\eqref{prob}, any Casimir function (of class $\mathcal{C}^2$) is preserved along the exact solutions. Similarly, for the undamped version of the stochastic Poisson system~\eqref{prob1}, the Hamiltonian is preserved along solutions. Such properties do not hold in general for damped versions, however one has the following  result when the Casimir and Hamiltonian functions are assumed to be homogeneous.


\begin{proposition}\label{prop:Casimir}
Consider the linearly damped stochastic Poisson system~\eqref{prob}. Assume that $C$ is a Casimir function of class $\mathcal{C}^2$,
which is an homogeneous function of degree $p\in(0,\infty)$.
Then, the mapping $C$ is a conformal Casimir function for the linearly damped stochastic Poisson system~\eqref{prob}:
if $\bigl(y(t)\bigr)_{t\in[0,\mathbf{T})}$ denotes the solution to~\eqref{prob},
almost surely one has
\[
C(y(t))=\exp\left( -p\int_0^t\gamma(s)\,\diff s  \right) C(y_0),\quad\forall~t\in[0,\mathbf{T}).
\]
Consider the linearly damped stochastic Poisson system with one noise~\eqref{prob1}.
Assume that the Hamiltonian function $H$ is of class $\mathcal C^2$  and is an homogeneous function of degree $q\in(0,\infty)$.
Then, one has the following energy balance:
if $\bigl(y(t)\bigr)_{t\in[0,\mathbf{T})}$ denotes the solution to~\eqref{prob1}, almost surely one has
$$
H(y(t))=\exp\left(-p\int_0^t\gamma(s)\,\diff s  \right) H(y_0),\quad\forall~t\in[0,\mathbf{T}).
$$
\end{proposition}
The proof of this result follows and generalises the computations done in \cite[Section~3.2]{MR4242161}, see also \cite[Theorem~1]{MR3898821}.
\begin{proof}
If $F$ is a Casimir function for the system~\eqref{prob}, then by definition one has $\nabla F^T B\equiv0$. If $F$ is the Hamiltonian function for the system~\eqref{prob1}, one has $\nabla F^{T}B \nabla F\equiv0$.
By the product rule for Stratonovich SDEs and the definition of the considered problems, one then obtains the identity
\[
\diff F(y(t))=-\gamma(t)\nabla F(y(t))^T y(t)\diff t.
\]

The function $F$ is assumed to be an homogeneous function of degree $p$, which yields the identity
\[
\diff F(y(t))=-p\gamma(t) F(y(t))\diff t.
\]
Therefore, for all $t\in[0,\mathbf{T})$, one has almost surely
\[
F(y(t))=\exp\left(-p\int_0^t\gamma(s)\,\diff s\right)F(y_0).
\]
The proof is completed.
\end{proof}

Proposition~\ref{prop:Casimir} above can be used to show almost sure global existence and boundedness of the exact solutions to linearly damped stochastic Poisson systems~\eqref{prob} and~\eqref{prob1}, under appropriate conditions.
\begin{corollary}\label{cor:exist}
Let $F$ be an homogeneous function of degree $p$. Moreover, assume that $F$ is either a Casimir function of class $\mathcal{C}^2$ in the case of the SDE~\eqref{prob} or a Hamiltonian function of class $\mathcal{C}^2$ in the case of the one-noise SDE~\eqref{prob1}.
Assume that ${\bf m}(F)=\underset{y\in\R^d;~\|y\|=1}\min~F(y)>0$ is positive.

Then almost surely one has $\mathbf{T}=\infty$. Furthermore, for any initial value $y_0\in\R^d$ and any final time $T\in(0,\infty)$, there exists a positive real number $R(\gamma,T,y_0,F)\in(0,\infty)$ such that for all $t\in[0,T]$, one has almost surely
\[
\|y(t)\|\le \frac{F(y_0)^{\frac1p}}{{\bf m}(F)^{\frac1p}}\exp\left(\int_0^T|\gamma(s)|\,\diff s\right)\leq R(\gamma,T,y_0,F)<\infty.
\]
\end{corollary}

In the case of quadratic functionals $C(y)=\frac12 y^TDy$, the condition ${\bf m}(C)>0$ is satisfied if and only if the symmetric matrix $D$ is positive definite.
This holds for the rigid body system described Example~\ref{exprigid} and Remark~\ref{rem:CasH}.

\begin{proof}
Recall that for $F$ an homogeneous function of degree $p$, one has $F(y)=\|y\|^p F\left(\frac{y}{\|y\|}\right)\ge {\bf m}(F)\|y\|^p$. Combining the condition ${\bf m}(F)>0$ and the upper bound obtained in Proposition~\ref{prop:Casimir}, for all $t\in[0,\mathbf{T}\wedge T)$ one has almost surely
\[
\|y(t)\|^p\le \frac{F(y(t))}{{\bf m}(F)}\le \frac{F(y_0)}{{\bf m}(F)}\exp\left(p\int_0^T|\gamma(s)|\,\diff s\right)= R(\gamma,T,y_0,F).
\]
Using the upper bound above, it is straightforward to prove that $\mathbf{T}=\infty$ almost surely. This concludes the proof.
\end{proof}

\begin{remark}\label{remark:exist}
Proposition~\ref{prop:Casimir} and Corollary~\ref{cor:exist} can be applied to the damped stochastic rigid body system from Example~\ref{exprigid}, which thus admits a unique global solution. The damped stochastic Maxwell--Bloch system from Example~\ref{expmax} also admits a unique global solution: Corollary~\ref{cor:exist} cannot be applied.
However Proposition~\ref{prop:Casimir} provides almost sure upper bounds on $C(y(t))=\frac12\|y_2(t)\|^2+\frac12\|y_3(t)\|^2$ and one then observes that
the damped stochastic Maxwell--Bloch system can then be treated as an SDE with globally Lipschitz continuous coefficients.
The damped stochastic pendulum system~\eqref{pendulum} from Example~\ref{exppend} and the damped stochastic three-dimensional Poisson system~\eqref{caointro} from Example~\ref{expcao} have globally Lipschitz coefficients and thus a unique global solution.
\end{remark}

\section{Stochastic conformal exponential integrator}\label{sec-scheme}

In this section, we present a numerical method for approximating solutions to the linearly damped stochastic Poisson systems~\eqref{prob}~and~\eqref{prob1}.
We show that the proposed stochastic exponential integrator is conformal quadratic Casimir for the system~\eqref{prob}.
Furthermore, the numerical integrator satisfies an energy balance for the SDE systems~\eqref{prob1} when the Hamiltonian is homogeneous.
Finally, we prove that the proposed integrator converges strongly with order $1/2$ for the system~\eqref{prob} and with order $1$ for the system~\eqref{prob1},
and weakly with order $1$ for the systems~\eqref{prob} and~\eqref{prob1}, under appropriate conditions.

\subsection{Description of the integrator}

Let $T\in(0,\infty)$ denote an arbitrary final time and $N$ be a positive integer. Define the time-step size $\dt=T/N$. 
Set $t_n=n\dt$ for $n=0,1,\ldots, N$ and $t_{n+\frac12}=t_n+\frac{\dt}{2}=\frac{t_n+t_{n+1}}{2}$ for all $n\in\{0,\ldots,N-1\}$. Define the Wiener increments $\Delta W_{m,n}=W_m(t_{n+1})-W_m(t_n)$
for $n\in\{0,1,\ldots, N-1\}$ and $m\in\{1,\ldots,M\}$. The numerical scheme is implicit and it is thus required to consider truncated Wiener increments $\widehat{\Delta W_{m,n}}$, defined as follows: given an integer $k\in\N$, introduce the threshold $A_{\dt,k}=\sqrt{2k|\log(\dt)|}$, the auxiliary function
\[
\chi_{\dt,k}(\zeta)=
\begin{cases}
\frac{\zeta}{|\zeta|}\min\bigl(|\zeta|,A_{\dt,k}\bigr),\quad \zeta\neq 0,\\
0,\quad \zeta=0,
\end{cases}
\]
and set
\begin{equation}\label{eq:What}
\widehat{\Delta W_{m,n}}=\sqrt{\dt}\chi_{\dt,k}\left(\frac{\Delta W_{m,n}}{\sqrt{\dt}}\right).
\end{equation}
We need below the following properties of the truncated Wiener increments. First, by construction one has the almost sure upper bound $|\widehat{\Delta W_{m,n}}|\le |\Delta W_{m,n}|$, therefore the truncated Wiener increments inherit moment bounds properties from the standard Wiener increments $\Delta W_{m,n}$. Moreover, one has the almost sure upper bound $|\widehat{\Delta W_{m,n}}|\le \sqrt{\dt}A_{\dt,k}$.
Finally, for all $p\in[1,\infty)$, there exists a $C_{k,p}\in(0,\infty)$ such that one has 
\begin{equation}\label{wtrunc}
\bigl(\E[|\Delta W_{m,n}-\widehat{\Delta W_{m,n}}|^{p}]\bigr)^{\frac{1}{p}}\le C_{k,p}\dt^{\frac{1+k}{2}},
\end{equation}
and such that one has
\begin{equation}\label{wtrunc2}
0\le \E[\Delta W_{m,n}^2]-\E[\widehat{\Delta W_{m,n}}^2]\le (1+A_{\dt,k})\dt^{k+1}.
\end{equation}
See for instance~\cite[Section~2.1]{MR1951908} for further details on truncated Wiener increments.

Let us recall that the discrete gradient $\overline{\nabla}H:\R^d\times\R^d\to\R^d$ associated with an Hamiltonian function $H:\R^d\to\R$ is defined by
\begin{equation}\label{discretegradient}
\displaystyle\overline{\nabla}H(z_0,z_1)=\int_0^1\nabla H\bigl((1-\eta)z_0+\eta z_1\bigr)\,\text d\eta,\quad \forall~z_0,z_1\in\R^d,
\end{equation}
see for instance~\cite{MR1411343,MR2451073} and references therein.

Finally, for all $n\in\{0,\ldots,N-1\}$, set
\begin{equation}\label{defX01}
X_n^0=\int_{t_{n+1/2}}^{t_{n}}\gamma(s)\diff s\quad\text{and}\quad X_n^1=\int_{t_{n+1/2}}^{t_{n+1}}\gamma(s)\diff s.
\end{equation}

The proposed numerical scheme is inspired by the numerical methods for ODE studied in \cite{MR4242161} and by the energy-preserving schemes
for stochastic Poisson systems studied in \cite{MR3210739}. The stochastic conformal exponential integrator for the linearly damped stochastic Poisson system~\eqref{prob}
is defined by
\begin{equation}\label{confexp}
\begin{aligned}
{e^{X_{n}^1}y_{n+1}-e^{X_n^0}y_n}&=B\left(\frac{e^{X_n^0}y_n+e^{X_{n}^1}y_{n+1}}{2}\right)\overline{\nabla}H_0(\e^{X_n^0}y_n,\e^{X_{n}^1}y_{n+1})\dt \\
&+B\left(\frac{e^{X_n^0}y_n+e^{X_{n}^1}y_{n+1}}{2}\right)
\sum_{m=1}^M\overline{\nabla}H_m(\e^{X_n^0}y_n,\e^{X_{n}^1}y_{n+1})\widehat{\Delta W_{m,n}},
\end{aligned}
\end{equation}
for all $n\in\{0,\ldots,N-1\}$, where $y_0\in\R^d$ is an arbitrary given initial value.

When considering the linearly damped stochastic Poisson system~\eqref{prob1}, the proposed integrator is given by
\begin{equation}\label{confexp1}
{e^{X_{n}^1}y_{n+1}-e^{X_n^0}y_n}=B\left(\frac{e^{X_n^0}y_n+e^{X_{n}^1}y_{n+1}}{2}\right)\left\{ \overline{\nabla}H(\e^{X_n^0}y_n,\e^{X_{n}^1}y_{n+1})
\bigl(\dt+c\widehat{\Delta W_{m,n}}\bigr)\right\}.
\end{equation}



Observe that the numerical scheme~\eqref{confexp} has the following equivalent formulation, using auxiliary variables $\widehat{z_n},z_{n+1}\in\R^d$:
\begin{equation}\label{split}
\left\lbrace
\begin{aligned}
\widehat{z_n}&=\exp\bigl(-\int_{t_n}^{t_{n+\frac12}}\gamma(s)\,\text ds\bigr)y_n,\\
z_{n+1}&=\widehat{z_n}+B\left(\frac{\widehat{z_n}+z_{n+1}}{2}\right)\left( \overline{\nabla}H_0(\widehat{z_n},z_{n+1})\dt+
\sum_{m=1}^M\overline{\nabla}H_m(\widehat{z_n},z_{n+1})\widehat{\Delta W_{m,n}}\right),\\
y_{n+1}&=\exp\bigl(-\int_{t_{n+\frac12}}^{t_{n+1}}\gamma(s)\,\text ds\bigr)z_{n+1}.
\end{aligned}
\right.
\end{equation}
Therefore the stochastic conformal exponential integrator~\eqref{confexp} can be interpreted as a splitting integrator, where a Strang splitting strategy is applied, for the subsystem
\[
\diff y(t)=-\gamma(t)y(t)\diff t
\]
which is solved exactly on the intervals $[t_n,t_{n+\frac12}]$ and $[t_{n+\frac12},t_{n+1}]$, and the subsystem
\[
\diff z(t)=B(z(t))\nabla H_0(z(t))\diff t+\sum_{m=1}^M B(z(t))\nabla H_m(z(t))\circ\,\diff W_m(t)
\]
which is solved approximately on the interval $[t_n,t_{n+1}]$ with $z(t_n)=\widehat{z_n}$ by an extension of the energy-preserving scheme from \cite{MR3210739}.

The numerical integrator~\eqref{confexp} and the equivalent formulation~\eqref{split} are implicit schemes, it is thus necessary to justify that it admits a unique solution, assuming that the time-step size $\dt$ is sufficiently small. This is performed in two steps: first when the SDE~\eqref{prob} has globally Lipschitz drift and diffusion coefficients, second after studying the qualitative properties of the numerical solutions to obtain almost sure upper bounds.

\begin{lemma}\label{lem:scheme}
Assume that either
\begin{itemize}
\item the structure matrix $B$ is constant and for all $m\in\{0,1,\ldots,M\}$ the Hamiltonian function $H_m$ is of class $\mathcal{C}^2$ with bounded second order derivatives;
\item[] \hspace{-0.5cm}or
\item the structure matrix $B$ is of class $\mathcal{C}^1$, is bounded and has bounded first order derivative and for all $m\in\{0,1,\ldots,M\}$ the Hamiltonian function $H_m$ is
of class $\mathcal{C}^2$ with bounded first and second order derivatives.
\end{itemize}
Recall that $k$ denotes the integer used to define truncated Wiener increments.
There exists $\overline{\dt}_k>0$ such that if $\dt\in(0,\overline{\dt}_k)$ then the numerical scheme~\eqref{confexp} admits a unique solution.
\end{lemma}
The well-posedness of the numerical scheme~\eqref{confexp1} follows from Lemma~\ref{lem:scheme} with a straightforward modification of the assumptions.
\begin{proof}
Let us consider the equivalent formulation~\eqref{split} of the numerical scheme~\eqref{confexp}. It suffices to show that for $n\in\{0,\ldots,N-1\}$, given $\widehat{z}_n\in\R^d$ and truncated Wiener increments $\widehat{\Delta W_{m,n}}$ for $m\in\{1,\ldots,M\}$, there exists a unique solution $z=z_{n+1}$ to the fixed point equation
\begin{equation}\label{fixpoint}
z=\widehat{z_n}+B\left(\frac{\widehat{z_n}+z}{2}\right)\left( \overline{\nabla}H_0(\widehat{z_n},z)\dt+
\sum_{m=1}^M\overline{\nabla}H_m(\widehat{z_n},z)\widehat{\Delta W_{m,n}}\right).
\end{equation}
A sufficient condition is to prove that the auxiliary mapping
\[
z\mapsto\psi(z;\widehat{z}_n,\tau,\widehat{\Delta W_{1,n}},\ldots,\widehat{\Delta W_{M,n}})=\widehat{z_n}+B\left(\frac{\widehat{z_n}+z}{2}\right)\left( \overline{\nabla}H_0(\widehat{z_n},z)\dt+
\sum_{m=1}^M\overline{\nabla}H_m(\widehat{z_n},z)\widehat{\Delta W_{m,n}}\right)
\]
is a contraction if $\tau\in(0,\overline{\tau}_k)$, i.\,e. that there exists $C_{k,\tau}\in(0,\infty)$ such that
\[
\big|\psi(z;\widehat{z}_n,\tau,\widehat{\Delta W_{1,n}},\ldots,\widehat{\Delta W_{M,n}})-\psi(z';\widehat{z}_n,\tau,\widehat{\Delta W_{1,n}},\ldots,\widehat{\Delta W_{M,n}})\big|\le C_{k,\tau}|z-z'|,\quad\forall~z,z'\in\R^d.
\]
This follows from the following elementary computations: for all $z,z'\in\R^d$ one has
\begin{align*}
\big|\psi(z;\widehat{z}_n,\tau,\widehat{\Delta W_{1,n}},\ldots,\widehat{\Delta W_{M,n}})&-\psi(z';\widehat{z}_n,\tau,\widehat{\Delta W_{1,n}},\ldots,\widehat{\Delta W_{M,n}})\big|\\
&\le \dt\big|\bigl[B\left(\frac{\widehat{z_n}+z}{2}\right)-B\left(\frac{\widehat{z_n}+z'}{2}\right)\bigr]\overline{\nabla}H_0(\widehat{z_n},z)\big|\\
&~+ \sum_{m=1}^M\big|\bigl[B\left(\frac{\widehat{z_n}+z}{2}\right)-B\left(\frac{\widehat{z_n}+z'}{2}\right)\bigr]\overline{\nabla}H_m(\widehat{z_n},z)\widehat{\Delta W_{m,n}}\big|\\
&~+\dt\big|B\left(\frac{\widehat{z_n}+z'}{2}\right)\bigl[\overline{\nabla}H_0(\widehat{z_n},z)-\overline{\nabla}H_0(\widehat{z_n},z')\big|\\
&~+\sum_{m=1}^M\big|B\left(\frac{\widehat{z_n}+z'}{2}\right)\bigl[\overline{\nabla}H_m(\widehat{z_n},z)-\overline{\nabla}H_m(\widehat{z_n},z')\bigr]\widehat{\Delta W_{m,n}}\big|\\
&\le C_{k,\tau}|z-z'|
\end{align*}
with $C_{k,\tau}$ given by
\begin{align*}
C_{k,\tau}&=\tau\Bigl(\frac12\|B'\|_{\infty} \|\nabla H_0\|_{\infty}+\|B\|_{\infty}\|\nabla^2 H_0\|_{\infty}\Bigr)\\
&~+\sqrt{2k\tau|\log(\tau)|}\sum_{m=1}^{M}\Bigl(\frac12\|B'\|_{\infty} \|\nabla H_m\|_{\infty}+\|B\|_{\infty}\|\nabla^2 H_m\|_{\infty}\Bigr).
\end{align*}
Observe that $C_{k,\tau}\to 0$ when $\tau\to 0$, then choosing $\overline{\tau}_k$ such that $C_{k,\tau}<1$ for all $\tau\in(0,\overline{\tau}_k)$ concludes the proof of the contraction property of the auxiliary mapping.

Applying the fixed point theorem and a recursion argument then shows the well-posedness of the implicit numerical scheme~\eqref{confexp} if
the time-step size is sufficiently small $\tau\in(0,\overline{\tau}_k)$. The proof is thus completed.
\end{proof}

Note that Lemma~\ref{lem:scheme} can be applied to the damped stochastic pendulum system~\eqref{pendulum} described in Example~\ref{exppend} and for the damped stochastic three-dimensional Poisson system described in Example~\ref{expcao}.


%

\subsection{Qualitative properties}

In this section we study the qualitative behavior of the numerical solution defined by~\eqref{confexp}. First, we show that the proposed numerical scheme
satisfies the conformal Casimir property of the exact solution of~\eqref{prob} given in Proposition~\ref{prop:Casimir}, when quadratic Casimir functions are considered. Second, we show that for the system~\eqref{prob1}, then one obtains the same energy balance as in Proposition~\ref{prop:Casimir}, when the Hamiltonian function is homogeneous of degree $p$.

To establish the qualitative results above, it is not necessary to assume that the numerical schemes~\eqref{confexp} and~\eqref{confexp1} have unique solutions.
In fact, like for the exact solutions, under appropriate conditions, such qualitative properties can be employed to show that there exists
a unique solution to the numerical scheme (see Corollary~\ref{cor:num} below).

\begin{remark}
Note that, in general, a Casimir function can be an arbitrary function.
However, already in the undamped and deterministic setting,
it is know that only linear and quadratic invariants can be preserved automatically by a numerical scheme,
see for instance \cite{MR2840298}. Hence, our focus is on numerical discretisations of the SDE~\eqref{prob}
preserving the property of conformal quadratic Casimirs. For non-quadratic Casimir functions,
one should exploit the structure of a specific problem to derive conformal numerical schemes.
\end{remark}


For the proofs below, the formulation~\eqref{split} of the numerical scheme~\eqref{confexp} is used.

\begin{proposition}\label{prop:confCasimir}
Consider the linearly damped stochastic Poisson system~\eqref{prob}. Assume that $C$ is a quadratic conformal Casimir function.

The stochastic exponential integrator~\eqref{confexp} is conformal quadratic Casimir: for any value of the time-step size $\tau=T/N$, if $\bigl(y_n\bigr)_{0\le n\le N}$ is a numerical solution given by~\eqref{confexp}, for all $n\in\{0,\ldots,N-1\}$, almost surely one has
almost surely
\[
C(y_{n+1})=\exp\left( -2\int_{t_n}^{t_{n+1}}\gamma(s)\diff s  \right) C(y_n).
\]
\end{proposition}

\begin{proof}
Consider the equivalent formulation~\eqref{split} of the scheme~\eqref{confexp}. 
Since $C$ is a quadratic mapping, one has
\[
C(z_{n+1})-C(\widehat{z_n})=\nabla C\left(\frac{z_{n+1}+\widehat{z_n}}2\right)^T(z_{n+1}-\widehat{z_n}).
\]
Using the definition of a Casimir functional, one has the identity
\[
\nabla C\left(\frac{z_{n+1}+\widehat{z_n}}2\right)^TB\left(\frac{z_{n+1}+\widehat{z_n}}2\right)=0,
\]
therefore applying the formulation~\eqref{split} of the numerical scheme, for all $n\in\{0,\ldots,N-1\}$ one thus obtains
\[
C(z_{n+1})-C(\widehat{z_n})=0.
\]

In addition, $C$ is a quadratic function, therefore it is an homogeneous function of degree $2$, and by~\eqref{split}, for all $n\in\{0,\ldots,N-1\}$ one has almost surely
\begin{align*}
C(y_{n+1})&=\exp\bigl(-2\int_{t_{n+\frac12}}^{t_{n+1}}\gamma(s)\text ds\bigr)C(z_{n+1})\\
&=\exp\bigl(-2\int_{t_{n+\frac12}}^{t_{n+1}}\gamma(s)\text ds\bigr)C(\widehat{z_n})\\
&=\exp\bigl(-2\int_{t_{n+\frac12}}^{t_{n+1}}\gamma(s)\text ds\bigr)\exp\bigl(-2\int_{t_n}^{t_{n+\frac12}}\gamma(s)\text ds\bigr)C(y_{n})\\
&=\exp\bigl( -2\int_{t_n}^{t_{n+1}}\gamma(s)\diff s  \bigr) C(y_n).
\end{align*}
The proof is thus completed.

%
\end{proof}

\begin{proposition}\label{prop:balanceNum}
Consider the linearly damped stochastic Poisson system~\eqref{prob1}. Assume that the Hamiltonian function $H$ is homogeneous of degree $p$.

The stochastic conformal exponential integrator~\eqref{confexp1} satisfies an almost sure energy balance:
for any value of the time-step size $\tau=T/N$, if $\bigl(y_n\bigr)_{0\le n\le N}$ is a numerical solution given by~\eqref{confexp1},
for all $n\in\{0,\ldots,N-1\}$ one has almost surely
\[
H(y_{n+1})=\exp\left( -p\int_{t_n}^{t_{n+1}}\gamma(s)\diff s  \right)H(y_n).
\]
\end{proposition}

\begin{proof}
Consider the equivalent formulation~\eqref{split} of the scheme~\eqref{confexp}.

Using the definition~\eqref{discretegradient} of the discrete gradient $\overline{\nabla}H$, one has
\[
H(z_{n+1})-H(\widehat{z_n})=\overline{\nabla}H(\widehat{z_n},z_{n+1})^T (z_{n+1}-z_n).
\]
Since the structure matrix $B(z)$ is skew-symmetric for all $z\in\R^d$, one has the identity
\[
\overline{\nabla}H(\widehat{z_n},z_{n+1})^T B\left(\frac{z_{n+1}+\widehat{z_n}}2\right)  \overline{\nabla}H(\widehat{z_n},z_{n+1})=0,
\]
therefore applying the formulation~\eqref{split} of the numerical scheme, for all $n\in\{0,\ldots,N-1\}$ one thus obtains
\[
H(z_{n+1})-H(\widehat{z_n})=0.
\]
%
Furthermore, assuming that the Hamiltonian function $H$ is an homogeneous function of degree $p$, by~\eqref{split}, for all $n\in\{0,\ldots,N-1\}$ one has almost surely
\begin{align*}
H(y_{n+1})&=\exp\bigl(-p\int_{t_{n+\frac12}}^{t_{n+1}}\gamma(s)\diff s\bigr)H(z_{n+1})
=\exp\bigl(-p\int_{t_{n+\frac12}}^{t_{n+1}}\gamma(s)\diff s\bigr)H(\widehat{z_{n}})\\
&=\exp\bigl(-p\int_{t_{n+\frac12}}^{t_{n+1}}\gamma(s)\diff s\bigr)\exp\bigl(-p\int_{t_n}^{t_{n+\frac12}}\gamma(s)\diff s\bigr)H(y_n)\\
&=\exp\bigl( -p\int_{t_n}^{t_{n+1}}\gamma(s)\diff s \bigr)H(y_n).
\end{align*}


The proof is thus completed.
\end{proof}

Under the same conditions as in Corollary~\ref{cor:exist} (existence of a unique global solution to the SDEs), the results from Propositions~\ref{prop:confCasimir} and~\ref{prop:balanceNum} provide almost sure upper bounds for numerical solutions given by~\eqref{confexp} and~\eqref{confexp1}. Moreover, such upper bounds can be applied to justify
that the numerical integrators admit unique solutions for sufficiently small time-step size.

\begin{corollary}\label{cor:num}
Consider the linearly damped stochastic Poisson system~\eqref{prob}, and assume that it admits a quadratic Casimir function $C(y)=\frac12 y^TDy$ with a symmetric positive definite matrix $D$.
Recall that $k$ denotes the integer used to define truncated Wiener increments.

There exists $\overline{\tau}_k(\gamma,T,y_0)>0$ such that if the time-step size satisfies $\tau\in(0,\overline{\tau}_k(\gamma,T,y_0))$, then there exists a unique solution $\bigl(y_n\bigr)_{0\le n\le N}$ to the numerical scheme~\eqref{confexp}.

Moreover, for all $n\in\{0,\ldots,N-1\}$ one has almost surely
\[
\|y_n\|\le \frac{\sqrt{C(y_0)}}{\sqrt{{\bf m}(C)}}\exp\bigl(\int_{0}^{T}|\gamma(s)|\,\diff s\bigr),
\]
where ${\bf m}(C)=\underset{y\in\R^d;~\|y\|=1}\min~C(y)>0$.

Consider the linarly damped stochastic Poisson system~\eqref{prob1}, and assume that the Hamiltonian $H$ is an homogeneous function of degree $p$ which satisfies
${\bf m}(H)=\underset{y\in\R^d;~\|y\|=1}\min~H(y)>0$.

There exists $\overline{\tau}_k(\gamma,T,y_0)>0$ such that if the time-step size satisfies $\tau\in(0,\overline{\tau}_k(\gamma,T,y_0))$, then there exists a unique solution $\bigl(y_n\bigr)_{0\le n\le N}$ to the numerical scheme~\eqref{confexp1}.

Moreover, for all $n\in\{0,\ldots,N-1\}$ one has almost surely
\[
\|y_n\|\le \frac{H(y_0)^{\frac{1}{p}}}{{\bf m}(H)^{\frac{1}{p}}}\exp\bigl(\int_{0}^{T}|\gamma(s)|\,\diff s\bigr).
\]
\end{corollary}

\begin{proof}
First, consider a solution $\bigl(y_n\bigr)_{0\le n\le N}$ to the numerical scheme~\eqref{confexp}, and assume that $C$ is a quadratic Casimir function for~\eqref{prob}. Owing to Proposition~\ref{prop:confCasimir}, for all $n\in\{0,\ldots,N\}$ one obtains the almost sure identity
\[
C(y_n)=\exp\left( -2\int_{0}^{t_{n}}\gamma(s)\diff s  \right)C(y_0)
\]
and thus the almost sure upper bound
\[
C(y_n)\le \exp\bigl(2\int_{0}^{T}|\gamma(s)|\,\diff s\bigr)C(y_0).
\]
Using the inequality
\[
C(y)\ge {\bf m}(C)\|y\|^2,\quad \forall~y\in\R^d
\]
then provides the first almost sure upper bound stated in Corollary~\ref{cor:num}.

Second, consider a solution $\bigl(y_n\bigr)_{0\le n\le N}$ to the numerical scheme~\eqref{confexp1}, and assume that the Hamiltonian function $H$ is homogeneous of degree $p$. Owing to Proposition~\ref{prop:balanceNum}, for all $n\in\{0,\ldots,N-1\}$ one obtains the almost sure identity
\[
H(y_n)=\exp\left(-p\int_{0}^{t_{n}}\gamma(s)\diff s  \right)H(y_0)
\]
and thus the almost sure upper bound
\[
H(y_n)\le \exp\bigl(p\int_{0}^{T}|\gamma(s)|\,\diff s\bigr)H(y_0).
\]
Using the inequality
\[
H(y)\ge {\bf m}(H)\|y\|^p,\quad \forall~y\in\R^d
\]
then provides the second almost sure upper bound stated in Corollary~\ref{cor:num}.

Finally, note that due to the almost sure upper bounds
\[
\underset{0\le n\le N}\sup~\|y_n\|\le R(\gamma,T,y_0)
\]
on any numerical solution $\bigl(y_n\bigr)_{0\le n\le N}$ to the numerical schemes~\eqref{confexp} and~\eqref{confexp1},
it can be assumed that the structure matrix $B$ has a bounded first order derivative and that the Hamiltonian functions $H_0,H_1,\ldots,H_M$ have bounded second order derivatives. One can then apply the result of Lemma~\ref{lem:scheme} to justify the existence of $\overline{\tau}_k(\gamma,T,y_0)>0$ such that the numerical scheme~\eqref{confexp} admits a unique solution if the time-step size satisfies the condition $\tau\in(0,\overline{\tau}_k(\gamma,T,y_0))$.
The proof is thus completed.
\end{proof}

\begin{remark}
Proposition~\ref{prop:confCasimir} and Corollary~\ref{cor:num} can be applied for the damped stochastic rigid body problem from Example~\ref{exprigid}:
the numerical scheme~\eqref{confexp} admits a unique solution which satisfies the same almost sure upper bounds as the exact solution of~\eqref{prob} in that example.
Concerning the damped stochastic Maxwell--Bloch system from Example~\ref{expmax}, the numerical scheme~\eqref{confexp} also admits a unique solution:
Corollary~\ref{cor:num} cannot be applied. However, Proposition~\ref{prop:confCasimir} provides almost sure upper bounds on $C(y_n)$ and the result of Lemma~\ref{lem:scheme} can then be  applied. We refer to Remark~\ref{remark:exist} for a similar discussion concerning the exact solutions of the damped stochastic rigid body and Maxwell--Bloch systems.
\end{remark}



\section{Strong and weak convergence results}\label{sec:convergence}

We now state and prove the strong and weak error estimates for the stochastic conformal exponential integrator~\eqref{confexp}
when applied to the linearly damped stochastic Poisson systems~\eqref{prob}~and~\eqref{prob1}.

To perform the convergence analysis, we assume that the following conditions hold, concerning the systems~\eqref{prob} or~\eqref{prob1}
and the numerical schemes~\eqref{confexp} or~\eqref{confexp1}.
\begin{assumption}\label{ass:prob}
Let $\bigl(y(t)\bigr)_{t\in[0,T]}$ denote the solution to the linearly damped stochastic Poisson system~\eqref{prob} and given the time-step size $\dt=T/N$ let $\bigl(y_n\bigr)_{n\in\{0,\ldots,N\}}$ denote the solution to~\eqref{confexp}.

There exists a positive real number $R(T,y_0)\in(0,\infty)$ and an integer $N(T,y_0)\in\N$ such that almost surely one has
\[
\underset{t\in[0,T]}\sup~\|y(t)\|+\underset{N\ge N(T,y_0)}\sup~\underset{n\in\{0,\ldots,N\}}\max~\|y_n\|\le R(T,y_0).
\]
\end{assumption}

\begin{assumption}\label{ass:prob1}
Let $\bigl(y(t)\bigr)_{t\in[0,T]}$ denote the solution to the linearly damped stochastic Poisson system~\eqref{prob1} and given the time-step size $\dt=T/N$ let $\bigl(y_n\bigr)_{n\in\{0,\ldots,N\}}$ denote the solution to~\eqref{confexp1}.

There exists a positive real number $R(T,y_0)\in(0,\infty)$ and an integer $N(T,y_0)\in\N$ such that almost surely one has
\[
\underset{t\in[0,T]}\sup~\|y(t)\|+\underset{N\ge N(T,y_0)}\sup~\underset{n\in\{0,\ldots,N\}}\max~\|y_n\|\le R(T,y_0).
\]
\end{assumption}

Notice that, owing to Corollary~\ref{cor:exist} and to Corollary~\ref{cor:num}, Assumption~\ref{ass:prob} is satisfied if the linearly damped stochastic Poisson system~\eqref{prob} admits a quadratic Casimir function $C(y)=\frac12y^TDy$ with a symmetric positive definite matrix $D$, and Assumption~\ref{ass:prob1} is satisfied if the Hamiltonian function $H$ is homogeneous of degree $p$ and satisfies ${\bf m}(H)=\underset{y\in\R^d;~\|y\|=1}\min~H(y)>0$.

First, we show that, under Assumption~\ref{ass:prob}, the numerical scheme~\eqref{confexp} is of strong order $1/2$ when applied to the linearly damped stochastic Poisson system~\eqref{prob}.
\begin{theorem}\label{th:strong}
Consider the solution $\bigl(y(t)\bigr)_{t\in[0,T]}$ to the linearly damped stochastic Poisson system~\eqref{prob} and the solution $\bigl(y_n\bigr)_{n\in\{0,\ldots,N\}}$ to the numerical scheme~\eqref{confexp} with time-step size $\dt=T/N$ and with the truncated Wiener increments defined by~\eqref{eq:What} with $k\ge 1$.

Let Assumption~\ref{ass:prob} be satisfied. Moreover, assume that the structure matrix $B$ is of class $\mathcal{C}^4$, and that for all $m\in\{0,1,\ldots,M\}$ the Hamiltonian functions $H_m$ are of class $\mathcal{C}^5$. Finally, assume that $\gamma$ is of class $\mathcal{C}^1$.

There exists $C(T,y_0)\in(0,\infty)$ and $N(T,y_0)\in\N$ such that for all $N\ge N(T,y_0)$ one has
\[
\underset{n\in\{0,\ldots,N\}}\sup~\bigl(\E[\|y_n-y(t_n)\|^2\bigr)^{\frac12}\le C(T,y_0)\dt^{\frac12}.
\]
\end{theorem}

Second, we show that, under Assumption~\ref{ass:prob}, the numerical scheme~\eqref{confexp} is of strong order $1$ when applied to the linearly damped stochastic Poisson
system~\eqref{prob} with $M=1$.
\begin{theorem}\label{th:strongM1}
Consider the solution $\bigl(y(t)\bigr)_{t\in[0,T]}$ to the linearly damped stochastic Poisson system~\eqref{prob} with one noise and the solution $\bigl(y_n\bigr)_{n\in\{0,\ldots,N\}}$ to the numerical scheme~\eqref{confexp} with time-step size $\dt=T/N$ and with the truncated Wiener increments defined by~\eqref{eq:What} with $k\ge 2$.

Let Assumption~\ref{ass:prob} be satisfied. Moreover, assume that the structure matrix $B$ is of class $\mathcal{C}^4$, and that the Hamiltonian functions $H_0$ and $H_1$ are of class
$\mathcal{C}^5$. Finally, assume that $\gamma$ is of class $\mathcal{C}^1$.

There exists $C(T,y_0)\in(0,\infty)$ and $N(T,y_0)\in\N$ such that for all $N\ge N(T,y_0)$ one has
\[
\underset{n\in\{0,\ldots,N\}}\sup~\bigl(\E[\|y_n-y(t_n)\|^2\bigr)^{\frac12}\le C(T,y_0)\dt.
\]
\end{theorem}

Third, we show that, under Assumption~\ref{ass:prob1}, the numerical scheme~\eqref{confexp1} is of strong order $1$ when applied to the linearly damped stochastic
Poisson system with one noise~\eqref{prob1}.
\begin{theorem}\label{th:strong1}
Consider the solution $\bigl(y(t)\bigr)_{t\in[0,T]}$ to the linearly damped stochastic Poisson system~\eqref{prob1} and the solution $\bigl(y_n\bigr)_{n\in\{0,\ldots,N\}}$ to the numerical scheme~\eqref{confexp1} with time-step size $\dt=T/N$ and with the truncated Wiener increments defined by~\eqref{eq:What} with $k\ge 2$.

Let Assumption~\ref{ass:prob1} be satisfied. Moreover, assume that the structure matrix $B$ is of class $\mathcal{C}^4$, that the Hamiltonian function $H$ is of class $\mathcal{C}^5$.
Finally, assume that $\gamma$ is of class $\mathcal{C}^1$.

There exists $C(T,y_0)\in(0,\infty)$ and $N(T,y_0)\in\N$ such that for all $N\ge N(T,y_0)$ one has
\[
\underset{n\in\{0,\ldots,N\}}\sup~\bigl(\E[\|y_n-y(t_n)\|^2\bigr)^{\frac12}\le C(T,y_0)\dt.
\]
\end{theorem}

Finally, we show that, the numerical scheme~\eqref{confexp} is of weak order $1$ when applied to the linearly damped stochastic Poisson system~\eqref{prob} with $M\ge 2$. Note that when $M=1$ the result also holds, and is a straightforward consequence of Theorem~\ref{th:strong1}.
\begin{theorem}\label{th:weak}
Consider the solution $\bigl(y(t)\bigr)_{t\in[0,T]}$ to the linearly damped stochastic Poisson system~\eqref{prob} with $M\geq2$ and the solution $\bigl(y_n\bigr)_{n\in\{0,\ldots,N\}}$ to the numerical scheme~\eqref{confexp} with time-step size $\dt=T/N$ and with the truncated Wiener increments defined by~\eqref{eq:What} with $k\ge 2$.

Let Assumption~\ref{ass:prob} be satisfied. Moreover, assume that the structure matrix $B$ is of class $\mathcal{C}^4$, and that for all $m\in\{0,1,\ldots,M\}$ the Hamiltonian functions $H_m$ are of class $\mathcal{C}^5$. Finally, assume that $\gamma$ is of class $\mathcal{C}^1$.

Let $\varphi\colon\mathbb R^d\to\mathbb R$ be a function of class $\mathcal C^4$.

There exists $C(T,y_0,\varphi)\in(0,\infty)$ and $N(T,y_0)\in\N$ such that for all $N\ge N(T,y_0)$ one has
\[
\left|\E[\varphi(y_N)]-\E[\varphi(y(T)) ] \right|\leq C(T,y_0,\varphi)\dt.
\]
\end{theorem}


The main steps for the proofs of convergence are inspired by \cite{MR4593213}: Due to the conformal properties of the proposed numerical methods, one can consider an auxiliary SDE with globally Lipschitz coefficients and bounded derivatives. For this auxiliary SDE, one can use standard techniques to prove strong and weak error estimates.
To do so, one performs Stratonovich--Taylor expansions of the exact and numerical solutions.
Finally, one compares these expansions in order to get local error estimates and prove the convergence results.

\subsection{Preliminary results}

Let us first introduce additional notation. In this section, the values of the initial value $y_0\in\R^d$ and of the time-step size $\tau=T/N$ are fixed. It is assumed that either Assumption~\ref{ass:prob} or~\ref{ass:prob1} is satisfied, and that $N\ge N(T,y_0)=N_0$.

For all $r\ge 0$, let $\mathcal{B}(0,r)=\{y\in\R^d:~\|y\|\le r\}$.

For all $m\in\{0,\ldots,M\}$, introduce the vector fields $f_m\colon \R^d\to\R^d$ given by
\[
f_m(y)=B(y)\nabla H_m(y),\quad \forall~y\in\R^d,
\]
and set
\[
g_0(t,y)=f_0(y)-\gamma(t)y+\frac12\sum_{m=1}^{M}f_m'(y) f_m(y),\quad \forall~(t,y)\in[0,T]\times\R^d.
\]
Observe that the linearly damped stochastic Poisson system~\eqref{prob} can be written as Stratonovich and It\^o SDEs
\begin{equation}\label{auxSDE}
\begin{aligned}
\diff y(t)&=\bigl(f_0(y(t))-\gamma(t)y(t)\bigr)\diff t+\sum_{m=1}^M  f_m(y(t))\circ\diff W_m(t)\\
&= g_0(t,y(t)) \diff t+\sum_{m=1}^M  f_m( y(t))\diff W_m(t),
\end{aligned}
\end{equation}
for $t\in[0,T]$, with initial value $y(0)=y_0$.


Let $R=R(T,y_0)$ be such that Assumption~\ref{ass:prob} or Assumption~\ref{ass:prob1} is satisfied. For any $y\in\mathcal{B}(0,R)$ 
and any $n\in\{0,\ldots,N-1\}$, define the stochastic process $\bigl(Y^{n,y}(t)\bigr)_{t\in[t_n,t_{n+1}]}$ which is solution to the Stratonovich SDE
\begin{equation}\label{auxSDEn}
\left\lbrace
\begin{aligned}
\diff Y^{n,y}(t)&=\bigl(f_0(Y^{n,y}(t))-\gamma(t)Y^{n,y}(t)\bigr)\diff t+\sum_{m=1}^M  f_m(Y^{n,y}(t))\circ\diff W_m(t),\quad \forall~t\in[t_n,t_{n+1}],\\
Y^{n,y}(t_n)&=y.
\end{aligned}
\right.
\end{equation}
Proceeding as in the proof of Proposition~\ref{prop:Casimir}, if $C$ is a quadratic Casimir function, one has almost surely
\[
C(Y^{n,y}(t))=\exp\bigl(-2\int_{t_n}^{t}\gamma(s)\,\diff s\bigr)C(y),\quad \forall~t\in[t_n,t_{n+1}].
\]
As a result, when Assumption~\ref{ass:prob} holds, one obtains the following almost sure upper bound
\[
\underset{t\in[t_n,t_{n+1}]}\sup~\|Y^{n,y}(t)\|\le \frac{{\bf M} (C)^{\frac12}}{{\bf m}(C)^{\frac12}}\exp\bigl(\int_{0}^{T}|\gamma(s)|\,\diff s\bigr)R=R'(T,y_0),\quad \forall~y\in \mathcal{B}(0,R),
\]
where ${\bf m}(C)=\underset{y\in\R^d;~\|y\|=1}\min~C(y)$ and ${\bf M}(C)=\underset{y\in\R^d;~\|y\|=1}\max~C(y)$.

Similarly, when considering the system~\eqref{prob1} (and thus taking $M=1$ and $H_0=H_1=H$ in the auxiliary SDE~\eqref{auxSDE}),
if the Hamiltonian function $H$ is homogeneous of degree $p$,
one has almost surely
\[
H(Y^{n,y}(t))=\exp\bigl(-p\int_{t_n}^{t}\gamma(s)\,\diff s\bigr)H(y),\quad \forall~t\in[t_n,t_{n+1}].
\]
As a result, when Assumption~\ref{ass:prob1} holds, one obtains the following almost sure upper bound
\[
\underset{t\in[t_n,t_{n+1}]}\sup~\|Y^{n,y}(t)\|\le \frac{{\bf M}(H)^{\frac{1}{p}}}{{\bf m}(H)^{\frac{1}{p}}}\exp\bigl(\int_{0}^{T}|\gamma(s)|\,\diff s\bigr)R=R'(T,y_0),\quad \forall~y\in \mathcal{B}(0,R),
\]
where ${\bf m}(H)=\underset{y\in\R^d;~\|y\|=1}\min~H(y)$ and ${\bf M}(H)=\underset{y\in\R^d;~\|y\|=1}\max~H(y)$.

Assuming that the structure matrix $B$ is of class $\mathcal{C}^4$ and that the Hamiltonian functions $H_0,H_1,\ldots,H_M$ are of class $\mathcal{C}^5$, imply that the vector fields $f_0,\ldots,f_M$ are of class $\mathcal{C}^4$ on the ball $\mathcal{B}(0,R')$. Those vector fields and their derivatives of order less than $4$ are thus bounded on the ball $\mathcal{B}(0,R')$. In particular, one can consider that the SDE~\eqref{auxSDE} (interpreted in its It\^o formulation) has globally Lipschitz continuous drift and diffusion coefficients.

The error analysis requires to identify the Stratonovich--Taylor expansion for the solution of~\eqref{auxSDEn}
given by Lemma~\ref{lem:expansionexact} below.
\begin{lemma}\label{lem:expansionexact}
There exists $C(R')\in(0,\infty)$ such that, for all $N\ge N_0$, all $n\in\{0,\ldots,N-1\}$ and all $t\in[t_n,t_{n+1}]$, one has
\begin{align}\label{taylorexact}
Y^{n,y}(t)-y&=(t-t_n)\bigl(f_0(y)-\gamma(t_n)y\bigr)+\sum_{m=1}^Mf_m(y)(W_m(t)-W_m(t_n))\nonumber\\
&\quad+\sum_{m,\ell=1}^M f_m'(y)f_{\ell}(y)\int_{t_n}^t\int_{t_n}^s\circ\diff W_{\ell}(r)\circ\diff W_m(s)+R_n^{\text{ex}}(t,y),
\end{align}
where the reminder term $R_n^{\text{ex}}(t,y)$ satisfies
\begin{align}
\bigl(\E[ \| R_n^{\text{ex}}(t,y) \|^2]    \bigr)^{\frac12}&\leq C(R')(t-t_n)^{\frac32}\leq C(R')\dt^{\frac32},\label{reminderexactstrong}\\
\|\E[ R_n^{\text{ex}}(t,y) ] \|&\leq C(R')(t-t_n)^{2}\leq C(R')\dt^{2},\label{reminderexactweak}
\end{align}
for all $n\in\{0,\ldots,N-1\}$, all $t\in[t_n,t_{n+1}]$ and all $y\in\mathcal{B}(0,R)$.
\end{lemma}
We refer for instance to~\cite[Section~5.6]{MR1214374} and~\cite{MR4593213} for a proof.

For all $n\in\{0,\ldots,N-1\}$ and all $y\in\mathcal{B}(0,R)$, define $Y^{n,y}_{n+1}$ as the solution to the
stochastic conformal exponential scheme~\eqref{confexp} at iteration $n+1$ with value $y$ at iteration $n$:
\begin{equation}\label{auxschemen}
\begin{aligned}
{e^{X_{n}^1}Y^{n,y}_{n+1}-e^{X_n^0}y}&=B\left(\frac{e^{X_n^0}y+e^{X_{n}^1}Y^{n,y}_{n+1}}{2}\right)\overline{\nabla}H_0(\e^{X_n^0}y,\e^{X_{n}^1}Y^{n,y}_{n+1})\dt \\
&+B\left(\frac{e^{X_n^0}y+e^{X_{n}^1}Y^{n,y}_{n+1}}{2}\right)
\sum_{m=1}^M\overline{\nabla}H_m(\e^{X_n^0}y,\e^{X_{n}^1}Y^{n,y}_{n+1})\widehat{\Delta W_{m,n}},
\end{aligned}
\end{equation}
where we recall that $X_n^0$ and $X_n^1$ are given by~\eqref{defX01}. Taking into account the equivalent formulation~\eqref{split} of the integrator~\eqref{confexp}, it is convenient to introduce the auxiliary random variables
\begin{align*}
\widehat{Z_{n}^{n,y}}=\exp\bigl(-\int_{t_n}^{t_{n+\frac12}}\gamma(s)\,\text ds\bigr)y\quad\text{and}\quad
Z_{n+1}^{n,y}=\exp\bigl(\int_{t_n+\frac12}^{t_{n+1}}\gamma(s)\,\text ds\bigr)Y_{n+1}^{n,y}.
\end{align*}
The random variable $Z_{n+1}^{n,y}$ is the unique solution of
\[
Z_{n+1}^{n,y}=\widehat{Z_{n}^{n,y}}+B\left(\frac{\widehat{Z_{n}^{n,y}}+Z_{n+1}^{n,y}}{2}\right)\left( \overline{\nabla}H_0(\widehat{Z_{n}^{n,y}},Z_{n+1}^{n,y})\dt+
\sum_{m=1}^M\overline{\nabla}H_m(\widehat{Z_{n}^{n,y}},Z_{n+1}^{n,y})\widehat{\Delta W_{m,n}}\right).
\]
Note that owing to Assumption~\ref{ass:prob} or Assumption~\ref{ass:prob1} one obtains the almost sure upper bounds
\[
\underset{N\in\N}\sup~\underset{n\in\{0,\ldots,N-1\}}\max~\|\widehat{Z_{n}^{n,y}}\|\le R' \quad\text{and}\quad \underset{N\in\N}\sup~\underset{n\in\{1,\ldots,N\}}\max~\|Z_{n+1}^{n,y}\|\le R',
\]
with $R'=R'(T,y_0)$ defined above. The numerical solution $Y_{n+1}^{n,y}$ has the following Stratonovich--Taylor expansion

\begin{lemma}\label{lem:expansionnum}
There exists $C(R')\in(0,\infty)$ such that, for all $N\ge N_0$ and all $n\in\{0,\ldots,N-1\}$, one has
\begin{align}\label{taylornum}
Y_{n+1}^{n,y}-y=\tau \bigl(f_0(y)-\gamma(t_n)y\bigr)+\sum_{m=1}^M f_m(y)\widehat{\Delta W_{m,n}}
+\frac12\sum_{m,\ell=1}^M f_m'(y)f_{\ell}(y)\widehat{\Delta W_{m,n}}\widehat{\Delta W_{\ell,n}}+R_n^{\text{num}}(y),
\end{align}
where the reminder term $R_n^{\text{num}}(y)$ satisfies
\begin{align}
\bigl(\E[ \| R_n^{\text{num}}(y) \|^2] \bigr)^{\frac12}&\leq C(R')\dt^{\frac32},\label{remindernumstrong}\\
\|\E[ R_n^{\text{num}}(y) ] \|&\leq C(R')\dt^{2},\label{remindernumweak}
\end{align}
for all $y\in\mathcal{B}(0,R)$ and for all $n\in\{0,\ldots,N-1\}$.
\end{lemma}

\begin{proof}
Owing to~\eqref{auxschemen}, one has the decomposition
\begin{equation}\label{decomp}
Y_{n+1}^{n,y}-y=\widehat{Z_{n}^{n,y}}-y+Z_{n+1}^{n,y}-\widehat{Z_{n}^{n,y}}+Y_{n+1}^{n,y}-Z_{n+1}^{n,y}.
\end{equation}

Let $\zeta=\widehat{Z_{n}^{n,y}}$, and recall that the auxiliary variable $Z^{n,\zeta}_{n+1}$ is the unique solution to the fixed point equation~\eqref{fixpoint}
\begin{align*}
z&=\zeta+B\left(\frac{\zeta+z}{2}\right)\left( \overline{\nabla}H_0(\zeta,z)\dt+\sum_{m=1}^M\overline{\nabla}H_m(\zeta,z)\widehat{\Delta W_{m,n}}\right)=\psi(z;\zeta,\tau,\widehat{\Delta W_{1,n}},\ldots,\widehat{\Delta W_{M,n}}).
\end{align*}
The main task in the proof is to obtain the following claim:
\begin{align}\label{claim}
Z_{n+1}^{n,y}-\zeta=\tau f_0(\zeta)+\sum_{m=1}^M f_m(\zeta)\widehat{\Delta W_{m,n}}
+\frac12\sum_{m,\ell=1}^M f_m'(\zeta)f_{\ell}(\zeta)\widehat{\Delta W_{m,n}}\widehat{\Delta W_{\ell,n}}+r_n^{\text{num}}(\zeta),
\end{align}
where the reminder term $r_n^{\text{num}}(\zeta)$ satisfies
\begin{align*}
\bigl(\E[ \| r_n^{\text{num}}(\zeta) \|^2] \bigr)^{\frac12}&\leq C(R')\dt^{\frac32},\\
\|\E[ r_n^{\text{num}}(\zeta) ] \|&\leq C(R')\dt^{2},
\end{align*}
for all $\zeta\in\mathcal{B}(0,R')$ and for all $n\in\{0,\ldots,N-1\}$.

For all $\zeta\in\mathcal{B}(0,R')$, all $\tau\in(0,\overline{\tau}_k(\gamma,T,y_0))$ and all $w_1,\ldots,w_m \in[-\sqrt{\tau}A_{\tau,k}, \sqrt{\tau}A_{\tau,k}]$, the unique solution of the fixed point equation
\[
z=\psi(z;\zeta,\tau,w_1,\ldots,w_M)
\]
can be written as
\[
z=\Phi(\zeta,\tau,w_1,\ldots,w_M).
\]
To simplify the notation, in the sequel the expression $\Phi(\zeta)\bigr\rvert_0$ means that the mapping $\Phi(\zeta)=\Phi(\zeta,\tau,w_1,\ldots,w_M)$ is evaluated with $\tau=0,w_1=0,\ldots,w_M=0$. The same notation is used for partial derivatives with respect to $\tau$ or to $w_1,\ldots,w_M$.

Owing to the local inversion theorem, the mapping $\Phi$ is of class $\mathcal{C}^4$, since the structure matrix $B$ is of class $\mathcal{C}^4$ and the Hamiltonian functions $H_0,H_1,\ldots,H_M$ are of class $\mathcal{C}^5$.

As a consequence, one has the following Taylor expansion of the mapping $\Phi$
\begin{align*}
\Phi(\zeta,\tau,w_1,\ldots,w_M)&=\Phi(\zeta)\bigr\rvert_0+\tau\partial_\tau\Phi(\zeta)\bigr\rvert_0+\sum_{m=1}^{M}w_m\partial_{w_m}\Phi(\zeta)\bigr\rvert_0
+\frac12\tau^2\partial_\tau^2\Phi(\zeta)\bigr\rvert_0\\
&+\tau\sum_{m=1}^{M}w_m\partial_\tau\partial_{w_m}\Phi(\zeta)\bigr\rvert_0
+\frac12\sum_{m_1,m_2=1}^{M}w_{m_1}w_{m_2}\partial_{w_{m_1}}\partial_{w_{m_2}}\Phi(\zeta)\bigr\rvert_0\\
&+\frac16\tau^3\partial_\tau^3\Phi(\zeta)\bigr\rvert_0+\frac12\tau^2\sum_{m=1}^{M}w_m\partial_\tau^2\partial_{w_m}\Phi(\zeta)\bigr\rvert_0\\
&+\frac12\tau\sum_{m_1,m_2=1}^{M}w_{m_1}w_{m_2}\partial_\tau\partial_{w_{m_1}}\partial_{w_{m_2}}\Phi(\zeta)\bigr\rvert_0
\\
&+\frac16\sum_{m_1,m_2,m_3=1}^{M}w_{m_1}w_{m_2}w_{m_3}\partial_{w_{m_1}}\partial_{w_{m_2}}\partial_{w_{m_3}}\Phi(\zeta)\bigr\rvert_0\\
&+\mathcal{R}(\zeta,\tau,w_1,\ldots,w_M),
\end{align*}
where there exists $C\in(0,\infty)$ such that for all $\zeta\in\mathcal{B}(0,R')$, all $\tau\in(0,\overline{\tau}_k(\gamma,T,y_0))$ and all $w_1,\ldots,w_m \in[-\sqrt{\tau}A_{\tau,k}, \sqrt{\tau}A_{\tau,k}]$, one has
\[
|\mathcal{R}(\zeta,\tau,w_1,\ldots,w_M)|\le C\bigl(\tau^4+\sum_{m=1}^{M}w_m^4\bigr).
\]

For the proof of the claim~\eqref{claim}, it is sufficient to compute the values of the function $\Phi$
\[
\Phi(\zeta)\bigr\rvert_0,
\]
of the first order derivatives
\[\partial_\tau\Phi(\zeta)\bigr\rvert_0, \partial_{w_m}\Phi(\zeta)\bigr\rvert_0,
\]
and of the second order derivatives
\[
\partial_{w_{m_1}}\partial_{w_{m_2}}\Phi(\zeta)\bigr\rvert_0.
\]
Indeed, all the other terms in the Taylor expansion of the mapping $\Phi$ are either of higher order than necessary or have expectation equal to $0$
when evaluated at $w_1=\widehat{\Delta W_{1,n}},\ldots,w_M=\widehat{\Delta W_{M,n}}$. Let us now compute these terms.

First, by definition of the fixed point equation, one has $\Phi(\zeta)\bigr\rvert_0=\Phi(\zeta,0,0,\ldots,0)=\zeta$.

To compute the first order derivative, using the definition of the mapping $\Phi$, writing
\[
\Phi(\zeta)=\zeta+B\left(\frac{\zeta+\Phi(\zeta)}{2}\right)\left( \overline{\nabla}H_0(\zeta,\Phi(\zeta))\dt+\sum_{m=1}^M\overline{\nabla}H_m(\zeta,\Phi(\zeta))w_m\right)
\]
and applying the first order derivative operator $\partial_\tau$ yield
\begin{align*}
\partial_\tau \Phi(\zeta)&=B\left(\frac{\zeta+\Phi(\zeta)}{2}\right)\overline{\nabla}H_0(\zeta,\Phi(\zeta))
+\frac{\tau}{2}B'\left(\frac{\zeta+\Phi(\zeta)}{2}\right)\partial_\tau \Phi(\zeta) \overline{\nabla}H_0(\zeta,\Phi(\zeta))\\
&+\tau B\left(\frac{\zeta+\Phi(\zeta)}{2}\right)\partial_{z_2}\overline{\nabla}H_0(\zeta,\Phi(\zeta))\partial_\tau \Phi(\zeta)
+\sum_{m=1}^{M}\frac{w_m}{2}B'\left(\frac{\zeta+\Phi(\zeta)}{2}\right)\partial_\tau \Phi(\zeta) \overline{\nabla}H_m(\zeta,\Phi(\zeta))\\
&+\sum_{m=1}^{M} w_m B\left(\frac{\zeta+\Phi(\zeta)}{2}\right)\partial_{z_2}\overline{\nabla}H_m(\zeta,\Phi(\zeta))\partial_\tau \Phi(\zeta).
\end{align*}
Evaluating the above equation at $\tau=0,w_1=0,\ldots,w_M=0$ and using the identity $\Phi(\zeta)\bigr\rvert_0=\zeta$ gives
\[
\partial_\tau \Phi(\zeta)\bigr\rvert_0=B(\zeta)\nabla H_0(\zeta).
\]
Similarly, for the other first order derivative of $\Phi$, one obtains
\begin{align*}
\partial_{w_m}\Phi(\zeta)&=\frac{\tau}{2}B'\left(\frac{\zeta+\Phi(\zeta)}{2}\right)\partial_{w_m} \Phi(\zeta) \overline{\nabla}H_0(\zeta,\Phi(\zeta))
+\tau B\left(\frac{\zeta+\Phi(\zeta)}{2}\right)\partial_{z_2}\overline{\nabla}H_0(\zeta,\Phi(\zeta))\partial_{w_m} \Phi(\zeta)\\
&+B\left(\frac{\zeta+\Phi(\zeta)}{2}\right)\overline{\nabla}H_m(\zeta,\Phi(\zeta))+\sum_{\ell=1}^{M}\frac{w_\ell}{2}B'\left(\frac{\zeta+\Phi(\zeta)}{2}\right)\partial_{w_m} \Phi(\zeta) \overline{\nabla}H_\ell(\zeta,\Phi(\zeta))\\
&+\sum_{\ell=1}^{M} w_\ell B\left(\frac{\zeta+\Phi(\zeta)}{2}\right)\partial_{z_2}\overline{\nabla}H_\ell(\zeta,\Phi(\zeta))\partial_{w_m} \Phi(\zeta).
\end{align*}
Evaluating the above equation at $\tau=0,w_1=0,\ldots,w_M=0$ and using the identity $\Phi(\zeta)\bigr\rvert_0=\zeta$ gives
\[
\partial_{w_m} \Phi(\zeta)\bigr\rvert_0=B(\zeta)\nabla H_m(\zeta).
\]
Let us now compute the second order derivatives $\partial_{w_{m_1}}\partial_{w_{m_2}}\Phi(\zeta)\bigr\rvert_0$. Since there are no other derivatives that need to be computed, it is not restrictive to eliminate all the terms that vanish when $\tau=0,w_1=0,\ldots,w_M=0$. Using the formula above for $\partial_{w_m}\Phi(\zeta)$, one obtains
\begin{align*}
\partial_{w_{m_1}}\partial_{w_{m_2}}\Phi(\zeta)\bigr\rvert_0&=\frac12 B'\left(\frac{\zeta+\Phi(\zeta)\bigr\rvert_0}{2}\right)\partial_{w_{m_1}}\Phi(\zeta)\bigr\rvert_0\overline{\nabla}H_{m_2}(\zeta,\Phi(\zeta)\bigr\rvert_0)\\
&+B\left(\frac{\zeta+\Phi(\zeta)\bigr\rvert_0}{2}\right)\partial_{z_2}\overline{\nabla}H_{m_2}(\zeta,\Phi(\zeta)\bigr\rvert_0)\partial_{w_{m_1}}\Phi(\zeta)\bigr\rvert_0\\
&+\frac12B'\left(\frac{\zeta+\Phi(\zeta)\bigr\rvert_0}{2}\right)\partial_{w_{m_2}}\Phi(\zeta)\bigr\rvert_0\overline{\nabla}H_{m_1}(\zeta,\Phi(\zeta)\bigr\rvert_0)\\
&+B\left(\frac{\zeta+\Phi(\zeta)\bigr\rvert_0}{2}\right)\partial_{z_2}\overline{\nabla}H_{m_1}(\zeta,\Phi(\zeta)\bigr\rvert_0)\partial_{w_{m_2}}\Phi(\zeta)\bigr\rvert_0\\
&=\frac12 B'(\zeta)\Bigl(B(\zeta)\nabla H_{m_1}(\zeta)\Bigr)\nabla H_{m_2}(\zeta)\\
&+\frac12 B(\zeta)\nabla^2 H_{m_2}(\zeta)\Bigl(B(\zeta)\nabla H_{m_1}(\zeta)\Bigr)\\
&+\frac12 B'(\zeta)\Bigl(B(\zeta)\nabla H_{m_2}(\zeta)\Bigr)\nabla H_{m_1}(\zeta)\\
&+\frac12 B(\zeta)\nabla^2 H_{m_1}(\zeta)\Bigl(B(\zeta)\nabla H_{m_2}(\zeta)\Bigr).
\end{align*}
In the computations above, the following property is used:
\[
\partial_{z_2}\overline{\nabla}H(z_1,z_2)=\int_0^1\eta\nabla^2 H\bigl((1-\eta)z_1+\eta z_2\bigr)\,\text d\eta
\]
and thus
\[
\partial_{z_2}\overline{\nabla}H(z,z)=\int_0^1\eta \,\text d\eta \nabla^2 H(z)=\frac12 \nabla^2 H(z).
\]
Observe that, using the notation $f_m(z)=B(z)\nabla H_m(z)$, the second order derivatives can be rewritten as
\[
\partial_{w_{m_1}}\partial_{w_{m_2}}\Phi(\zeta)\bigr\rvert_0=\frac12f_{m_2}'(\zeta)f_{m_1}(\zeta)+\frac12f_{m_1}'(\zeta)f_{m_2}(\zeta).
\]
In particular, when $m_1=m_2=m$, one obtains the expression
\[
\partial_{w_{m}}^2\Phi(\zeta)\bigr\rvert_0=f_{m}'(\zeta)f_{m}(\zeta).
\]

As a consequence of all the above computations, one obtains
\begin{align*}
Z^{n,y}_{n+1}&=\zeta+\tau f_0(\zeta)+\sum_{m=1}^{M}f_m(\zeta)\widehat{\Delta W_{m,n}}\\
&+\frac12 \sum_{m_1,m_2=1}^{M}\frac12\Bigl(f_{m_2}'(\zeta)f_{m_1}(\zeta)+f_{m_1}'(\zeta)f_{m_2}(\zeta)\Bigr)
\widehat{\Delta W_{m_1,n}}\widehat{\Delta W_{m_2,n}}+r_{n}^{\rm num}\\
&=\zeta+\tau f_0(\zeta)+\sum_{m=1}^{M}f_m(\zeta)\widehat{\Delta W_{m,n}}+\frac12 \sum_{m_1,m_2=1}^{M}f_{m_2}'(\zeta)f_{m_1}(\zeta)\widehat{\Delta W_{m_1,n}}\widehat{\Delta W_{m_2,n}}+r_n^{\rm num},
\end{align*}
where the reminder term $r_n^{\rm num}$ satisfies $\bigl(\E[\|r_n^{\rm num}\|^2]\bigr)^{\frac12}\le C\tau^{\frac32}$ and $\|\E[r_n^{\rm num}]\|\le C\tau^2$.

This concludes the proof of the claim~\eqref{claim}. Recalling the decomposition~\eqref{decomp}, one first has
\begin{equation}\label{err1}
\widehat{Z_{n}^{n,y}}-y=-\frac{\tau\gamma(t_n)}{2}y+r_n^{(1)},
\end{equation}
where one has
\[
|r_n^{(1)}(y)|\le C(R')\tau^2.
\]
Combining~\eqref{err1} and the claim~\eqref{claim} (with $\zeta=\widehat{Z_{n}^{n,y}}$), one thus obtains
\begin{align*}
Z^{n,y}_{n+1}-y&=Z^{n,y}_{n+1}-\widehat{Z_{n}^{n,y}}+\widehat{Z_{n}^{n,y}}-y\\
&=-\frac{\tau \gamma(t_n)}{2}y+\tau f_0(y)+\sum_{m=1}^M f_m(y)\widehat{\Delta W_{m,n}}
+\frac12\sum_{m,\ell=1}^M f_m'(y)f_{\ell}(y)\widehat{\Delta W_{m,n}}\widehat{\Delta W_{\ell,n}}+r_n^{(2)}(y),
\end{align*}
where one has
\begin{align*}
\bigl(\E[ \| r_n^{(2)}(y) \|^2] \bigr)^{1/2}&\leq C(R')\dt^{\frac32},\\
\|\E[ r_n^{(2)}(y) ] \|&\leq C(R')\dt^{2}.
\end{align*}
Finally, one has
\[
Y_{n+1}^{n,y}-Z_{n+1}^{n,y}=-\frac{\tau\gamma(t_n)}{2}Z_{n+1}^{n,y}+r_n^{(3)}(y)
\]
where the reminder verifies
\[
|r_n^{(3)}(y)|\le C(R')\tau^2.
\]
Using the expansion above and the decomposition~\eqref{decomp}, one obtains
\begin{align*}
Y_{n+1}^{n,y}-y&=Y_{n+1}^{n,y}-Z_{n+1}^{n,y}+Z_{n+1}^{n,y}-y\\
&=-\tau \gamma(t_n) y+\tau f_0(y)+\sum_{m=1}^M f_m(y)\widehat{\Delta W_{m,n}}
+\frac12\sum_{m,\ell=1}^M f_m'(y)f_{\ell}(y)\widehat{\Delta W_{m,n}}\widehat{\Delta W_{\ell,n}}+r_n^{(4)}(y),
\end{align*}
where one has
\begin{align*}
\bigl(\E[ \| r_n^{(4)}(y) \|^2] \bigr)^{\frac12}&\leq C(R')\dt^{\frac32},\\
\|\E[ r_n^{(4)}(y) ] \|&\leq C(R')\dt^{2}.
\end{align*}
Taking into account the identity $-\gamma(t_n)y+f_0(y)=g_0(t_n,y)$ then concludes the proof of~\eqref{taylornum}.
\end{proof}



%

\subsection{Proofs of the convergence results}



The strong convergence results are straightforward consequences of the fundamental theorem for mean-square convergence, see for instance~\cite[Theorem~1.1]{MR4369963}, combined with the Stratonovich--Taylor expansions given in Lemmas~\ref{lem:expansionexact} and~\ref{lem:expansionnum}. In all cases, owing to Assumptions~\ref{ass:prob} or~\ref{ass:prob1}, the exact and numerical solutions take values in a ball $\mathcal{B}(0,R)$ of radius $R=R(T,y_0)$ (depending on the time $T$ and on the initial value $y_0$), and as explained above the structure matrix $B$ and the Hamiltonian functions $H_0,H_1,\ldots,H_M$, and their derivatives, can be assumed to be bounded.

Before proceeding with the proof, let us compare the terms appearing at second order in the Stratonovich--Taylor expansions~\eqref{taylorexact} (with $t=t_{n+1}$) and~\eqref{taylornum}. On the one hand, observe that when $m=\ell$, for the exact solution one has
\[
f_m'(y)f_{m}(y)\int_{t_n}^{t_{n+1}}\int_{t_n}^s\circ\diff W_{m}(r)\circ\diff W_m(s)=\frac12 f_m'(y)f_m(y)\Delta W_{m,n}^2,
\]
whereas for the numerical solution one has
\[
\frac12f_m'(y)f_{m}(y)\widehat{\Delta W_{m,n}}\widehat{\Delta W_{m,n}}.
\]
Up to the error due to the truncation of the noise, which is given by~\eqref{wtrunc}, the two terms above thus match. On the other hand, when $m\neq \ell$, for the exact solution one has
\[
f_m'(y)f_{\ell}(y)\int_{t_n}^{t_{n+1}}\int_{t_n}^s\circ\diff W_{\ell}(r)\circ\diff W_m(s)\neq \frac12 f_m'(y)f_{\ell}(y){\Delta W_{m,n}}{\Delta W_{\ell,n}}
\]
whereas for the numerical solution one has
\[
\frac12f_m'(y)f_{\ell}(y)\widehat{\Delta W_{m,n}}\widehat{\Delta W_{\ell,n}}.
\]
As a consequence, the two terms do not match, even when the trunction of the noise is removed. This explains why the general case $M\in\N$ and the case $M=1$ are treated separately and different orders of convergence $1/2$ and $1$ are obtained.

\begin{proof}[Proof of Theorem~\ref{th:strong}]
Owing to the discussion above, in general the terms obtained when $m\neq \ell$ in the Stratonovich--Taylor expansions~\eqref{taylorexact} and~\eqref{taylornum} do not match and need to be taken into account in reminder terms.

Recall that the truncated Wiener increments $\widehat{\Delta W_{m,n}}$ are defined by~\eqref{eq:What} with $k\ge 1$. Owing to~\eqref{wtrunc}, the Stratonovich--Taylor expansion~\eqref{taylornum} of the numerical solution given in Lemma~\ref{lem:expansionnum} yields
\[
Y_{n+1}^{n,y}-y=\tau g_0(t_n,y)+\sum_{m=1}^M f_m(y){\Delta W_{m,n}}
+\frac12\sum_{m=1}^M f_m'(y)f_{m}(y){\Delta W_{m,n}}{\Delta W_{m,n}}+\rho_n^{\text{num}}(y)
\]
where the reminder term $\rho_n^{\text{num}}(y)$ satisfies
\[
\bigl(\E[ \| \rho_n^{\text{num}}(y) \|^2] \bigr)^{\frac12}\leq C\dt \quad \text{and} \quad
\|\E[ \rho_n^{\text{num}}(y) ] \|\leq C\dt^{\frac32}.
\]
Similarly, the Stratonovich--Taylor expansion~\eqref{taylorexact} of the exact solution (with $t=t_{n+1}$) given in Lemma~\ref{lem:expansionexact} yields
\[
Y^{n,y}(t_{n+1})-y=\tau g_0(t_n,y)+\sum_{m=1}^M f_m(y){\Delta W_{m,n}}
+\frac12\sum_{m=1}^M f_m'(y)f_{m}(y){\Delta W_{m,n}}{\Delta W_{m,n}}+\rho_n^{\text{ex}}(y)
\]
where the reminder term $\rho_n^{\text{ex}}(y)$ satisfies
\[
\bigl(\E[ \| \rho_n^{\text{ex}}(y) \|^2] \bigr)^{\frac12}\leq C\dt \quad \text{and} \quad
\|\E[ \rho_n^{\text{ex}}(y) ] \|\leq C\dt^{\frac32}.
\]
The application of the fundamental theorem for mean-square convergence~\cite[Theorem~1.1]{MR4369963} then yields the strong error estimates with mean-square order of convergence $1/2$ stated in Theorem~\ref{th:strong}.
\end{proof}

\begin{proof}[Proof of Theorem~\ref{th:strongM1}]
Assuming that $M=1$, the discussion above implies that one only needs to deal with the error due to the truncation of the Wiener increments, defined by~\eqref{eq:What}. Using the inequality~\eqref{wtrunc} with $k\ge 2$, 
the Stratonovich--Taylor expansion~\eqref{taylornum} of the numerical solution given in Lemma~\ref{lem:expansionnum} yields
\[
Y_{n+1}^{n,y}-y=\tau g_0(t_n,y)+f_1(y){\Delta W_{1,n}}
+\frac12 f_1'(y)f_{1}(y){\Delta W_{1,n}}{\Delta W_{1,n}}+\rho_n^{\text{num}}(y)
\]
where the reminder term $\rho_n^{\text{num}}(y)$ satisfies
\[
\bigl(\E[ \| \rho_n^{\text{num}}(y) \|^2] \bigr)^{\frac12}\leq C\dt^{\frac32} \quad \text{and} \quad
\|\E[ \rho_n^{\text{num}}(y) ] \|\leq C\dt^{2}.
\]
Similarly, the Stratonovich--Taylor expansion~\eqref{taylorexact} of the exact solution (with $t=t_{n+1}$) given in Lemma~\ref{lem:expansionexact} yields
\[
Y^{n,y}(t_{n+1})-y=\tau g_0(t_n,y)+f_1(y){\Delta W_{1,n}}
+\frac12 f_1'(y)f_{1}(y){\Delta W_{1,n}}{\Delta W_{1,n}}+\rho_n^{\text{ex}}(y)
\]
where the reminder term $\rho_n^{\text{ex}}(y)$ satisfies
\[
\bigl(\E[ \| \rho_n^{\text{ex}}(y) \|^2] \bigr)^{\frac12}\leq C\dt^{\frac32} \quad \text{and} \quad
\|\E[ \rho_n^{\text{ex}}(y) ] \|\leq C\dt^{2}.
\]
The application of the fundamental theorem for mean-square convergence~\cite[Theorem~1.1]{MR4369963} then yields the strong error estimates with mean-square order of convergence $1$ stated in Theorem~\ref{th:strongM1}.
\end{proof}

\begin{proof}[Proof of Theorem~\ref{th:strong1}]
It suffices to apply Theorem~\ref{th:strongM1} with $H_0=H_1=H$.
\end{proof}

Let us now turn to the proof of Theorem~\ref{th:weak} on weak error estimates. Notice that if $\phi:\R^d\to\R$ is a mapping of class $\mathcal{C}^4$, with bounded derivatives of order $1$ to $4$, then applying the Stratonovich--Taylor expansions~\eqref{taylorexact} (with $t=t_{n+1}$) and~\eqref{taylornum} one obtains for all $n\in\{0,\ldots,N-1\}$ the weak Taylor expansions of the exact and numerical solutions (i.\,e. Taylor expansions for the expected value of a functional $\phi(\cdot)$ applied to the exact and numerical solutions)
\begin{align*}
\E[\phi(Y^{n,y}(t_{n+1})]&=\phi(y)+\dt\nabla \phi(y)\cdot(f_0(y)-\gamma(t_n)y)\\
&+\sum_{m=1}^{M}\dt\nabla \phi(y)\cdot(f_m'(y) f_m(y))+\frac{1}{2}\sum_{m=1}^{M}\dt\nabla^2\phi(y)\cdot\bigl(f_m(y),f_m(y)\bigr)\\
&+\epsilon_{n}^{\text{ex}}(y),\\
\E[\phi(Y_{n+1}^{n,y})]&=\phi(y)+\dt\nabla \phi(y)\cdot(f_0(y)-\gamma(t_n)y)\\
&+\sum_{m=1}^{M}\E[\widehat{\Delta W_{m,n}}^2]\nabla \phi(y)\cdot(f_m'(y) f_m(y))+\frac{1}{2}\sum_{m=1}^{M}\E[\widehat{\Delta W_{m,n}}^2]\nabla^2\phi(y)\cdot\bigl(f_m(y),f_m(y)\bigr)\\
&+\epsilon_{n}^{\text{num}}(y),
\end{align*}
where the reminder terms $\epsilon_{n}^{\text{ex}}(y)$ and $\epsilon_{n}^{\text{num}}(y)$ satisfy
\[
\underset{0\le n\le N-1}\sup~\left(\epsilon_{n}^{\text{ex}}(y)+\epsilon_{n}^{\text{num}}(y)\right)\le C(\phi)\dt^2,
\]
with
\[
C(\phi)\le C\sum_{j=1}^{4}\underset{y\in\R^d}\sup~\|\nabla^j \phi(y)\|.
\]
for some $C\in(0,\infty)$.

Using the inequality~\eqref{wtrunc2} on the second moment of the truncated Wiener increments, one then obtains
\[
\big|\E[\phi(Y^{n,y}(t_{n+1})]-\E[\phi(Y_{n+1}^{n,y})]\big|\le C(\phi)(1+A_{\tau,k})\dt^{k+1}.
\]
The weak Taylor expansions of the exact and numerical solutions above are not used directly in the proof of Theorem~\ref{th:weak} below. Instead, an approach based on the expression and the decomposition of the weak error using the solution of an associated Kolmogorov equation is employed.

\begin{proof}[Proof of Theorem~\ref{th:weak}]
Let $T=t_N\in(0,\infty)$, define for all $t\in[0,T]$ and all $y\in \R^d$
\[
u^T(t,y)=\E[\varphi(y(T))|y(t)=y]=\E[\varphi(Y^{t,y}(T))],
\]
where $\bigl(Y^{t,y}(s)\bigr)_{s\ge t}$ is the solution to the It\^o SDE
\[
\left\lbrace
\begin{aligned}
\diff Y^{t,y}(s)&= g_0(t,Y^{t,y}(s)) \diff s+\sum_{m=1}^M  f_m(Y^{t,y}(s))\diff W_m(s),\quad \forall~s\ge t,\\
Y^{t,y}(t)&=y.
\end{aligned}
\right.
\]
The notation for the drift and diffusion coefficients is the same as in the SDE~\eqref{auxSDE}, however the initial value is imposed at time $t$ instead of time $0$. Proceeding as in the proof of Proposition~\ref{prop:Casimir}, if $y\in\mathcal{B}(0,R)$, where $R=R(T,y_0)$,
then one has $Y^{t,y}(s)\in\mathcal{B}(0,R')$ with $R'=R(T,y_0)$ given above.

The mapping $u^T$ is the solution to the backward Kolmogorov equation
\begin{equation}\label{kolmogo}
\left\lbrace
\begin{aligned}
&\frac{\partial}{\partial t}u^T(t,y)=\nabla_yu^T(t,y).g_0(t,y)+\frac12\sum_{m=1}^M\nabla_y^2u^T(t,y).(f_m(y),f_m(y)),\forall~(t,y)\in[0,T]\times\R^d\\
&u^T(T,y)=\varphi(y),\forall~y\in\R^d.
\end{aligned}
\right.
\end{equation}
If the function $\varphi$ is of class $C^4$, one can prove the following upper bounds on the derivatives of order $1$ to $4$ of $u^T(t,y)$ with respect to $y$: there exists $C_{R',T}(\phi)\inf(0,\infty)$ such that one has
\[
\sum_{j=1}^{4}\underset{t\in[0,T]}\sup~\underset{y\in\mathcal B(0,R)}\sup~\| \nabla^j u^T(t,y)\|\leq C_{R',T}(\varphi).
\]
Using the backward Kolmogorov equation~\eqref{kolmogo} one then obtains the following upper bounds on the first and second order derivatives of $u^T(t,y)$ with respect to $t$: one has
\[
\sum_{j=1}^{2}\underset{t\in[0,T]}\sup~\underset{y\in\mathcal B(0,R)}\sup~\| \frac{\partial^j}{\partial t^j} u^T(t,y)\|\leq C_{R',T}(\varphi).
\]
The definition of the mapping $u^T$ yields the following expression for the weak error:
\[
\E[\varphi(y_N)]-\E[\varphi(y(t_N))]=\E[u^T(t_N,y_N)]-\E[u^T(0,y_0)].
\]
Applying the standard telescoping sum argument then gives
\[
\E[\varphi(y_N)]-\E[\varphi(y(t_N))]=\sum_{n=0}^{N-1}\bigl(\E[u^T(t_{n+1},y_{n+1})]-\E[u^T(t_n,y_n)]\bigr).
\]
For all $n\in\{0,\ldots,N-1\}$, the local weak error term $\E[u^T(t_{n+1},y_{n+1})]-\E[u^T(t_n,y_n)]$ is decomposed into two contributions:
\begin{align}\label{locweak}
\E[u^T(t_{n+1},y_{n+1})]-\E[u^T(t_n,y_n)]&=\E[u^T(t_{n+1},y_n)]-\E[u^T(t_{n},y_n)]\nonumber\\
&+\E[u^T(t_{n+1},y_{n+1})]-\E[u^T(t_{n+1},y_n)].
\end{align}
For the first local weak error term appearing in~\eqref{locweak}, a second order Taylor expansion of $t\mapsto u^T(t,y_n)$ provides
\[
\E[u^T(t_{n+1},y_n)]-\E[u^T(t_{n},y_n)]=\dt\E[\frac{\partial}{\partial t}u^T(t_{n+1},y_n)]+\delta_n^1,
\]
where the reminder term satisfies
\[
|\delta_n^1|\leq C_{R',T}(\varphi)\dt^2,
\]
owing to the bounds on the first and second order temporal derivatives of $u^T$ stated above.

For the second local weak error term appearing in~\eqref{locweak}, using the weak Taylor expansion of the numerical solution with $y=y_n$ and $\phi=u^T(t_{n+1},\cdot)$ and the error bound~\eqref{wtrunc2} on the truncated Wiener increments, one obtains
\begin{align*}
\E[u^T(t_{n+1},y_{n+1})]-\E[u^T(t_{n+1},y_n)]&=\dt \E[\nabla_y u^T(t_{n+1},y_{n})].g_0(t_n,y_n)\\
&+\frac1{2}\dt\sum_{m=1}^M \E[\nabla_y^2u^T(t_{n+1},y_{n}).(f_m(y_n),f_m(y_n))]+\delta_n^2,
\end{align*}
where the reminder term satisfies
\[
|\delta_n^2|\leq C_{R',T}(\varphi)\left((1+A_{\tau,k})\dt^{k+1}+\dt^2 \right),
\]
owing to the bounds on the spatial derivatives of order $1$ to $4$ of $u^T$ stated above.

Choosing $k=2$ then gives the upper bound
\[
|\delta_n^2|\leq C_{R',T}(\varphi)\dt^2,
\]

Taking into account that $u^T$ is solution to the backward Kolmogorov equation~\eqref{kolmogo}, from the decomposition~\eqref{locweak} the local error term is thus given by
\[
\E[u^T(t_{n+1},y_{n+1})]-\E[u^T(t_n,y_n)]=\delta_n^1+\delta_n^2,
\]
and as a result one obtains
\begin{align*}
|\E[\varphi(y(t_N))]-\E[\varphi(y_N)]|&=|\E[u^T(0,y_0)]-\E[u^T(t_N,y_N)]|\\
&\leq \sum_{n=0}^{N-1}(|r_n^1|+|r_n^2|)\\
&\leq C_{R',T}(\varphi)\dt.
\end{align*}
This concludes the proof of Theorem~\ref{th:weak}.
\end{proof}

\section{Numerical experiments}\label{sec-numexp}
This section presents several numerical experiments in order to numerically confirm the above theoretical results
and in order to compare the stochastic conformal exponential integrator~\eqref{confexp} with two classical
numerical schemes for SDEs. When applied to the linearly damped stochastic Poisson system~\eqref{prob}, these two
numerical schemes are
\begin{itemize}
\item the Euler--Maruyama scheme (applied to the converted It\^o SDE with drift denoted by $\underline{B}(y_{n})\nabla\underline{H_0}(y_{n})$)
$$
y_{n+1}=y_n+\tau\Bigl(\underline{B}(y_{n})\nabla \underline{H_0}(y_{n})-\gamma(t_n)y_n\Bigr)+\sum_{m=1}^M B(y_{n})\nabla H_m(y_{n})\widehat{\Delta W_{m,n}}
$$
\item the stochastic midpoint scheme, see for instance \cite{MR1214374,MR4369963,MR3308418},
\begin{align*}
y_{n+1}&=y_n+\tau\Bigl(B(\frac{y_{n}+y_{n+1}}2)\nabla H_0(\frac{y_{n}+y_{n+1}}2)-\gamma(\frac{t_n+t_{n+1}}2)
\frac{y_{n}+y_{n+1}}2\Bigr)\\
&\quad+\sum_{m=1}^M B(\frac{y_{n}+y_{n+1}}2)\nabla H_m(\frac{y_{n}+y_{n+1}}2)\widehat{\Delta W_{m,n}}.
\end{align*}
\end{itemize}
These schemes will be denoted by {\sc Sexp}, {\sc EM}, and {\sc Midpoint} in all figures.
Note that the integral in the discrete gradient in the stochastic conformal exponential integrator~\eqref{confexp} can be computed
exactly for polynomial Hamiltonians for instance, else we use Matlab's \emph{integral} quadrature formula.

\subsection{A linearly damped stochastic mathematical pendulum}
As a first example, we consider the linearly damped stochastic mathematical pendulum from the introduction, see Example~\ref{exppend}:
\begin{equation}\label{pendul}
\diff
\begin{pmatrix}
y_1(t)\\
y_2(t)
\end{pmatrix}
=
\begin{pmatrix}
-\sin(y_2(t))\\
y_1(t)
\end{pmatrix}
\left(\diff t+c\circ\diff W(t)\right)-\gamma(t)y(t)\diff t,
\end{equation}
with $c=1$ and $\gamma(t)=2$.

For this problem we will only illustrate the strong rate of convergence of the stochastic conformal exponential integrator~\eqref{confexp}
as well as of the above described classical numerical integrators for SDEs. Note that the SDE~\eqref{pendul} has globally Lipschitz coefficients.
In the SDE~\eqref{pendul}, we set the end time $T=1$ and take the initial value $y(0)=(0.2,1)$.
The numerical schemes are applied with the time-step sizes $\dt=2^{-5},\ldots,2^{-12}$.
The reference solution is given by the stochastic conformal exponential integrator~\eqref{confexp}
with reference time-step size $\dt_{\text{ref}}=2^{-14}$. The expectation are approximated using $M_s=500$ independent Monte Carlo samples.
We have verified that this is enough for the Monte Carlo error
to be negligible. The strong rates of convergence of these time integrators are illustrated in Figure~\ref{fig:pendul}.

\begin{figure}[h]
\centering
\includegraphics[width=0.4\textwidth]{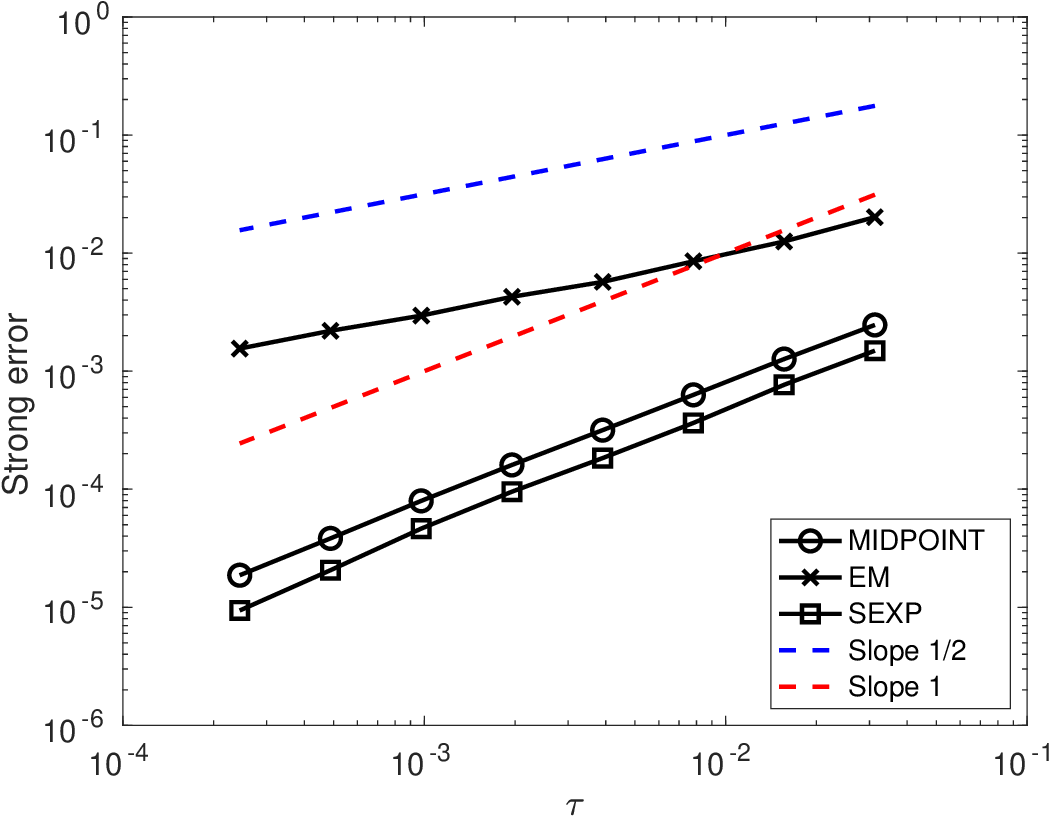}
\caption{Linearly damped stochastic mathematical pendulum~\eqref{pendul}: Strong convergence.}
\label{fig:pendul}
\end{figure}

The proposed exponential integrator has a strong order of convergence $1$, as stated in Theorem~\ref{th:strong1}.
The strong order of convergence of the Euler--Maruyama scheme is seen to be $1/2$, this is in accordance with the results from the literature
in this standard setting, see for instance \cite{MR1214374}. The stochastic midpoint scheme is known to have strong order $1$, see for instance \cite[Theorem~2.6]{MR1951908},
and this is the rate that is observed in the figure. We do not display plots for the weak errors since, in the present setting,
it is clear that the weak rates of convergence are $1$ for these numerical methods.

\subsection{Linearly damped free rigid body with random inertia tensor}
In this numerical experiment, we consider the free rigid body with random inertia tensor from \cite{MR4593213}, see also Example~\ref{exprigid}.
The damping function in the SDE~\eqref{prob} is $\gamma(t)=\frac12\cos(2t)$.
The considered linearly damped stochastic rigid body system thus reads
\begin{align}\label{eq:srb}
\diff\begin{pmatrix}y_1\\y_2\\y_3\end{pmatrix}=
B(y)
\left(\nabla H_0(y)\diff t+\nabla H_1(y)\circ\diff W_1(t)
+\nabla  H_2(y)\circ\diff W_2(t)+\nabla  H_3(y)\circ\diff W_3(t) \right)-\gamma(t)y\diff t,
\end{align}
where $y=(y_1,y_2,y_3)$, the skew-symmetric matrix
$$
B(y)=\begin{pmatrix}0 & -y_3 & y_2\\ y_3 & 0 & -y_1\\-y_2 & y_1 & 0\end{pmatrix},
$$
the Hamiltonian function
$$
H_0(y)=\frac12\left(\frac{y_1^2}{I_1}+\frac{y_2^2}{I_2}+\frac{y_3^2}{I_3}\right),
$$
and the diffusion functions
$H_k(y)=\frac{y_k^2}{{2}\widehat I_k}$, for $k=1,2,3$. Here, $I_k,\widehat I_k$, for $k=1,2,3$ are
positive and pairwise distinct real numbers called moment of inertia.
Observe that the linearly damped stochastic Poisson system~\eqref{eq:srb} has the conformal quadratic Casimir $C(y)=\frac12\left( y_1^2+y_2^2+y_3^2 \right)$.

In Figure~\ref{fig:trajSRB}, we display the quadratic Casimir $C(y)$ along the numerical solutions given by the Euler--Maruyama scheme, the stochastic midpoint scheme,
and the stochastic conformal exponential integrator~\eqref{confexp}. We use the following parameters:
the initial value reads $y(0)=(\cos(1.1),0,\sin(1.1))$, the moment of inertia are
$I_1=2, I_2=1, I_3=2/3$ and $\hat I_1=1, \hat I_2=2, \hat I_3=3$, the end time is $T=10$, the time-step size is $\dt=0.1$.
The preservation of this conformal quadratic Casimir by the stochastic exponential integrator~\eqref{confexp},
proved  in Proposition~\ref{prop:confCasimir}, is numerically illustrated in this figure.

\begin{figure}[h]
  \centering
  \includegraphics[width=0.4\textwidth]{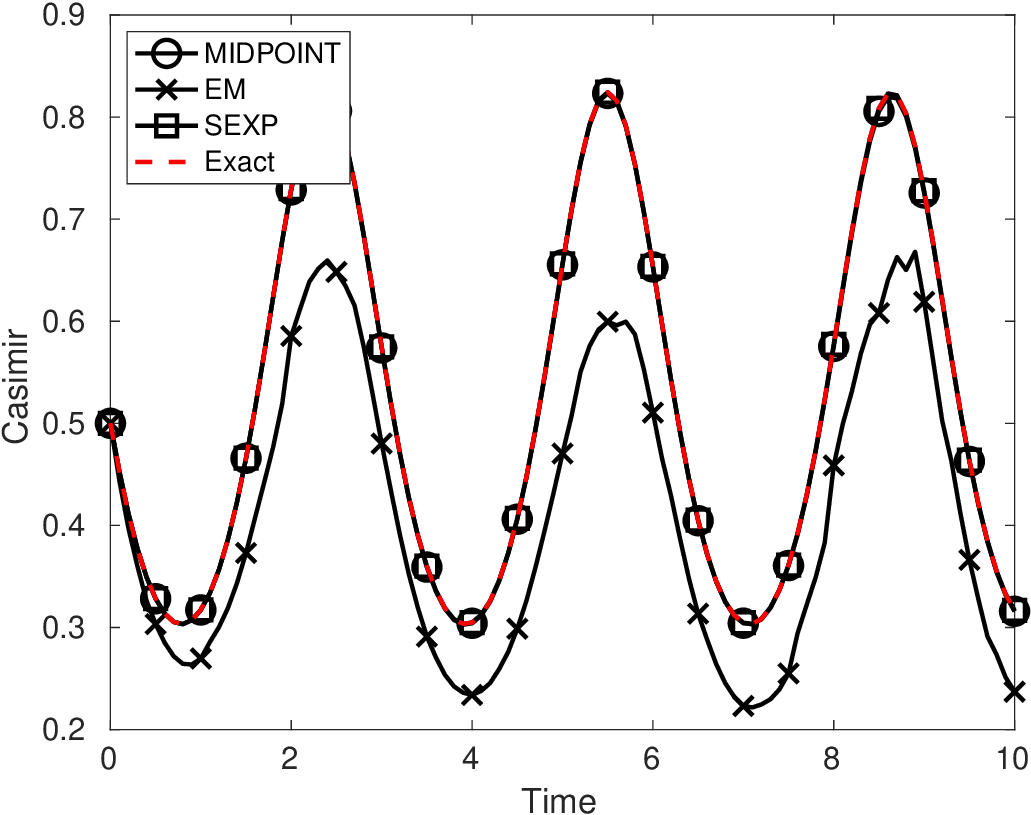}
  \caption{Linearly damped stochastic rigid body system~\eqref{eq:srb}: Evolution of the quadratic Casimir $C(y)=\frac12\left( y_1^2+y_2^2+y_3^2 \right)$.}
\label{fig:trajSRB}
\end{figure}

The strong convergence of the stochastic conformal exponential integrator~\eqref{confexp} is illustrated in Figure~\ref{fig:msSRB}.
To produce this figure, we use the same parameters as in the previous numerical experiment, except $T=1$. The numerical schemes are
applied with the range of time-step sizes $\dt=2^{-5},\ldots,2^{-14}$. The reference solution is given by stochastic conformal exponential integrator
with $\dt_{\text{ref}}=2^{-16}$. The expectations are approximated using $M_s=1000$ independent Monte Carlo samples. We have verified that this is enough for the Monte Carlo error
to be negligible. In this figure, one can observe the strong order of convergence $1$ in case of one noise (only $W_1$ in the SDE~\eqref{eq:srb}),
resp. order $1/2$ in case of several noises. This illustrates the results of Theorems~\ref{th:strong}~and~\ref{th:strongM1}.

\begin{figure}[h]
\centering
\begin{subfigure}{.4\textwidth}
  \centering
  \includegraphics[width=\textwidth]{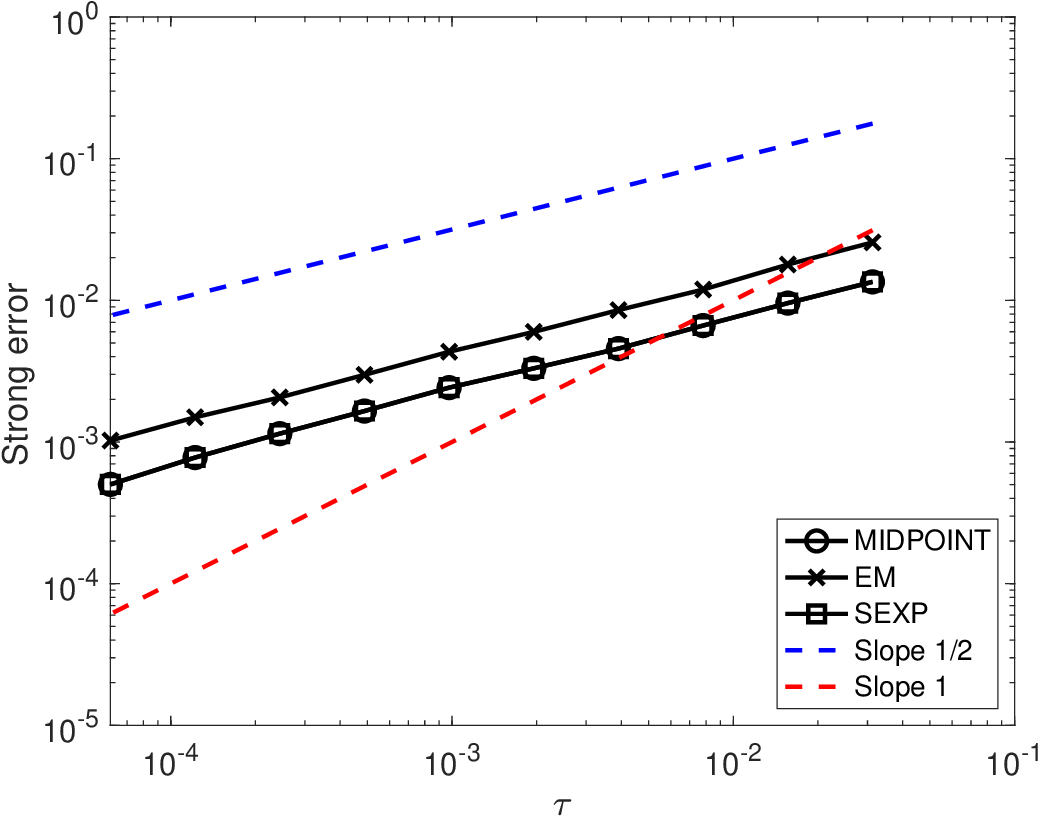}
  \caption{$3$ noises}
\end{subfigure}%
\begin{subfigure}{.4\textwidth}
  \centering
  \includegraphics[width=\textwidth]{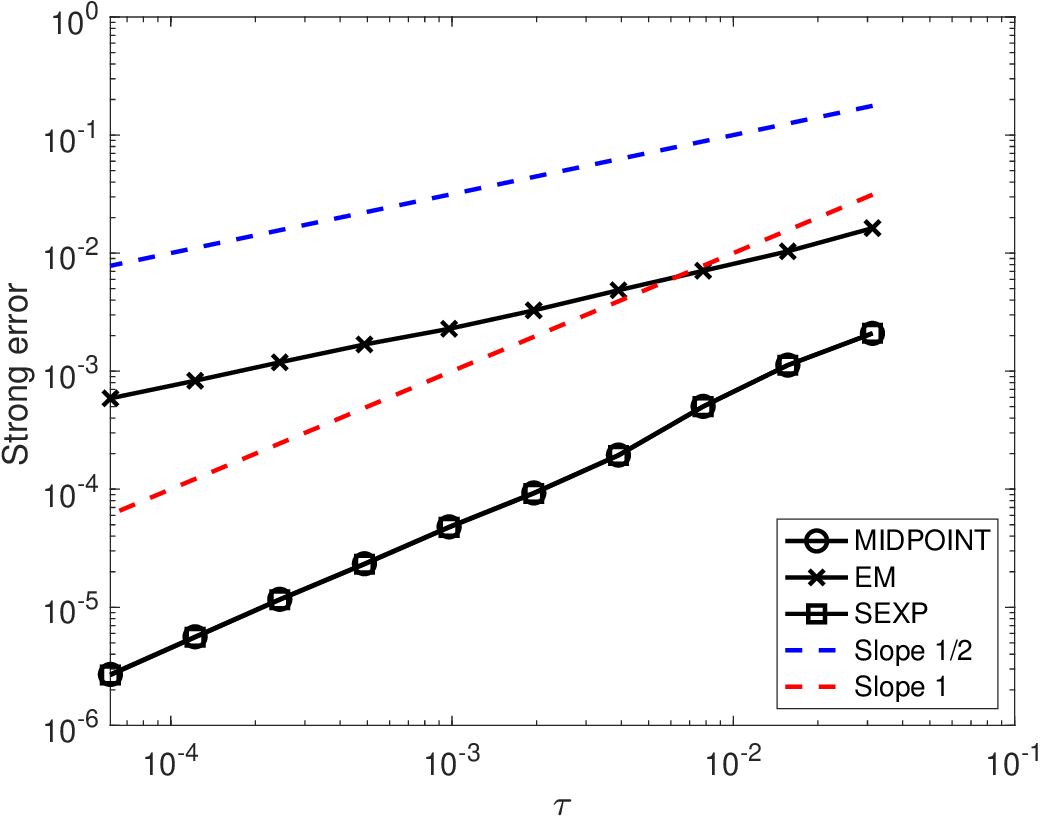}
  \caption{One noise (only $W_1$ in the SDE~\eqref{eq:srb})}
\end{subfigure}%

\caption{Linearly damped stochastic rigid body system~\eqref{eq:srb}: Strong convergence of the numerical schemes.}
\label{fig:msSRB}
\end{figure}

%
%

The weak convergence of the stochastic conformal exponential integrator~\eqref{confexp} stated in Theorem~\ref{th:weak}
is illustrated in Figure~\ref{fig:wSRB}. We use the initial value $y(0)=(\cos(1.1),0,\sin(1.1))$,
the final time $T=1$ and the moments of inertia as above. The stochastic exponential integrator is applied with the range of
time-step sizes $\dt=2^{-6},\ldots,2^{-12}$. The reference solutions are computed using this numerical scheme with $\dt_{\text{ref}}=2^{-12}$.
The expectation are approximated using $M_s=500000$ independent Monte Carlo samples. In this figure, one can observe
the weak order of convergence $1$ for the errors in several test functions. This illustrates the results of Theorem~\ref{th:weak}.

\begin{figure}[h]
\centering
\begin{subfigure}{.4\textwidth}
  \centering
  \includegraphics[width=\textwidth]{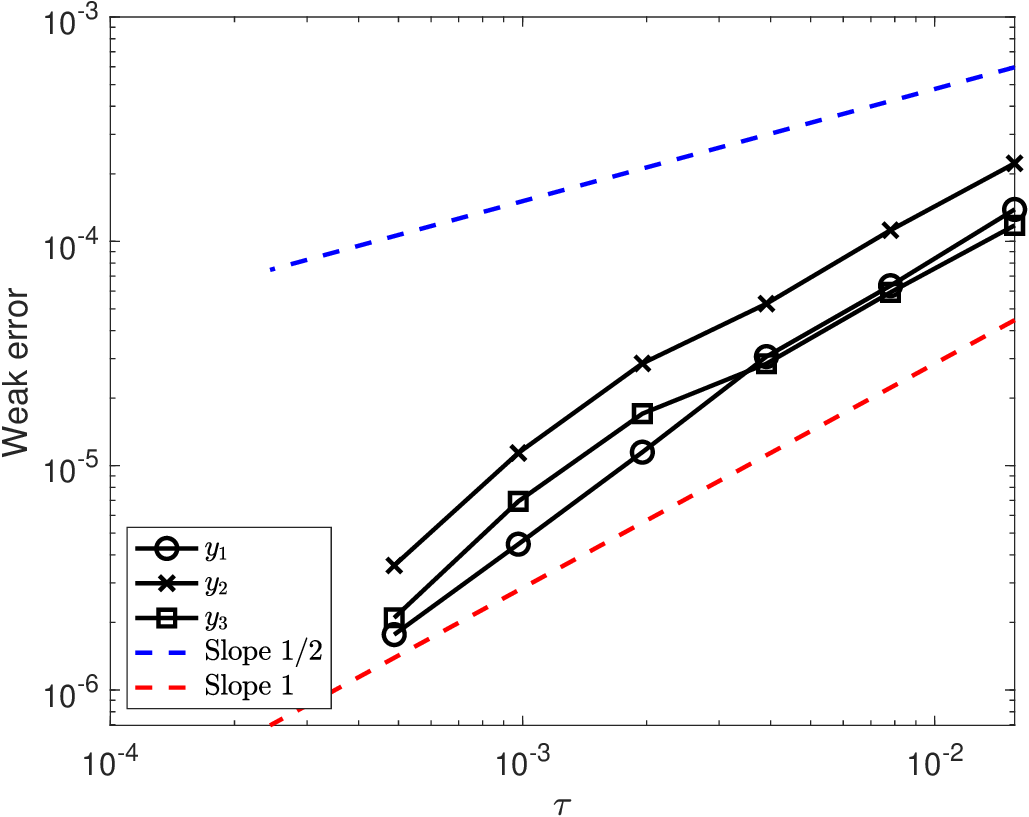}
  \caption{Weak error for $\mathbb E[y_j]$ for $j=1,2,3$}
\end{subfigure}%
\begin{subfigure}{.4\textwidth}
  \centering
  \includegraphics[width=\textwidth]{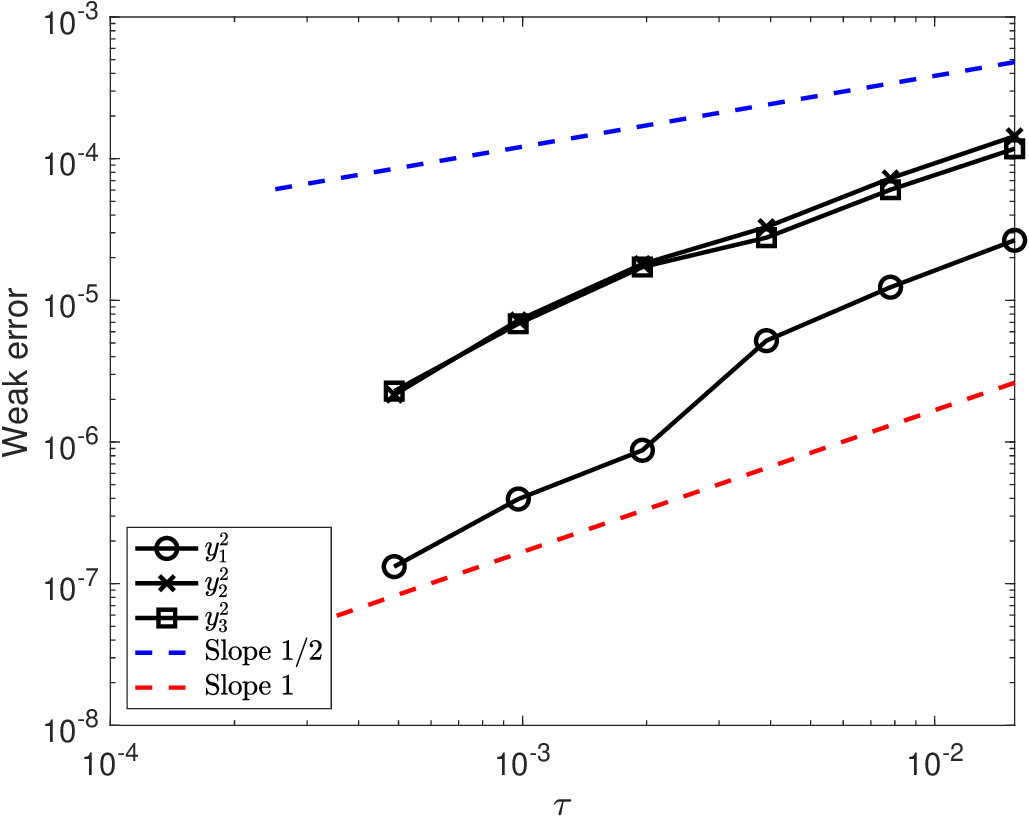}
  \caption{Weak error for $\mathbb E[y_j^2]$ for $j=1,2,3$}
\end{subfigure}\hfill
\begin{subfigure}{.4\textwidth}
  \centering
  \includegraphics[width=\textwidth]{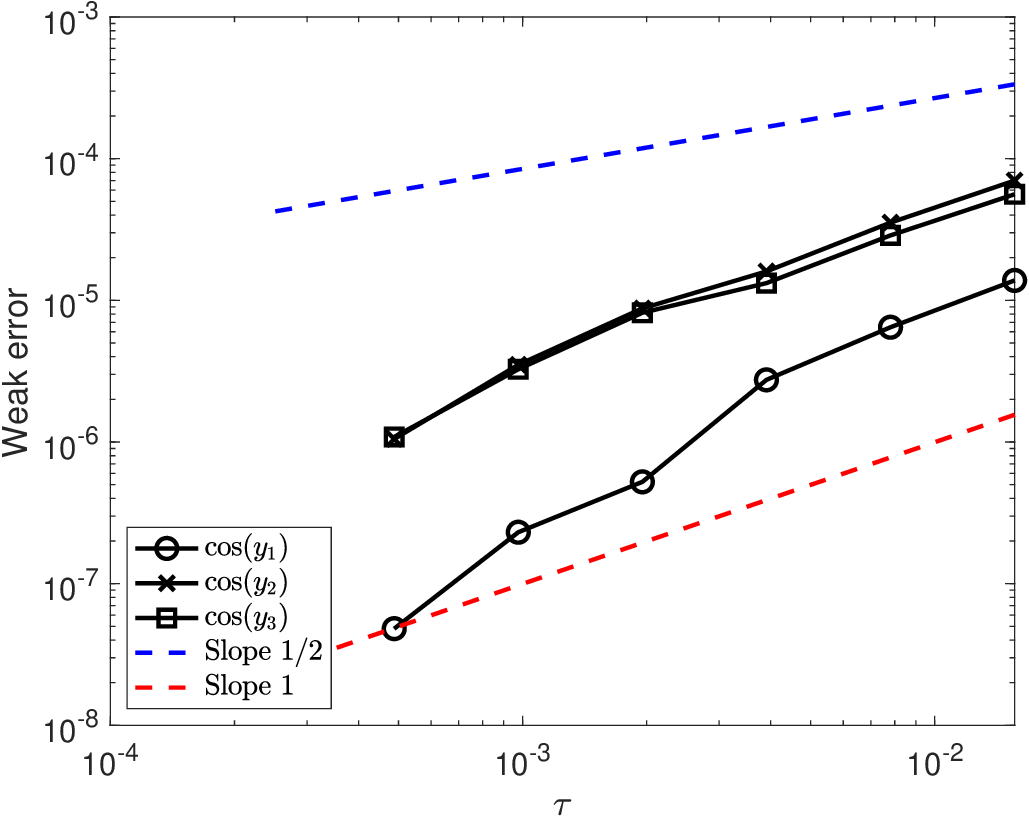}
  \caption{Weak error for $\mathbb E[\cos(y_j)]$ for $j=1,2,3$}
\end{subfigure}%
\begin{subfigure}{.4\textwidth}
  \centering
  \includegraphics[width=\textwidth]{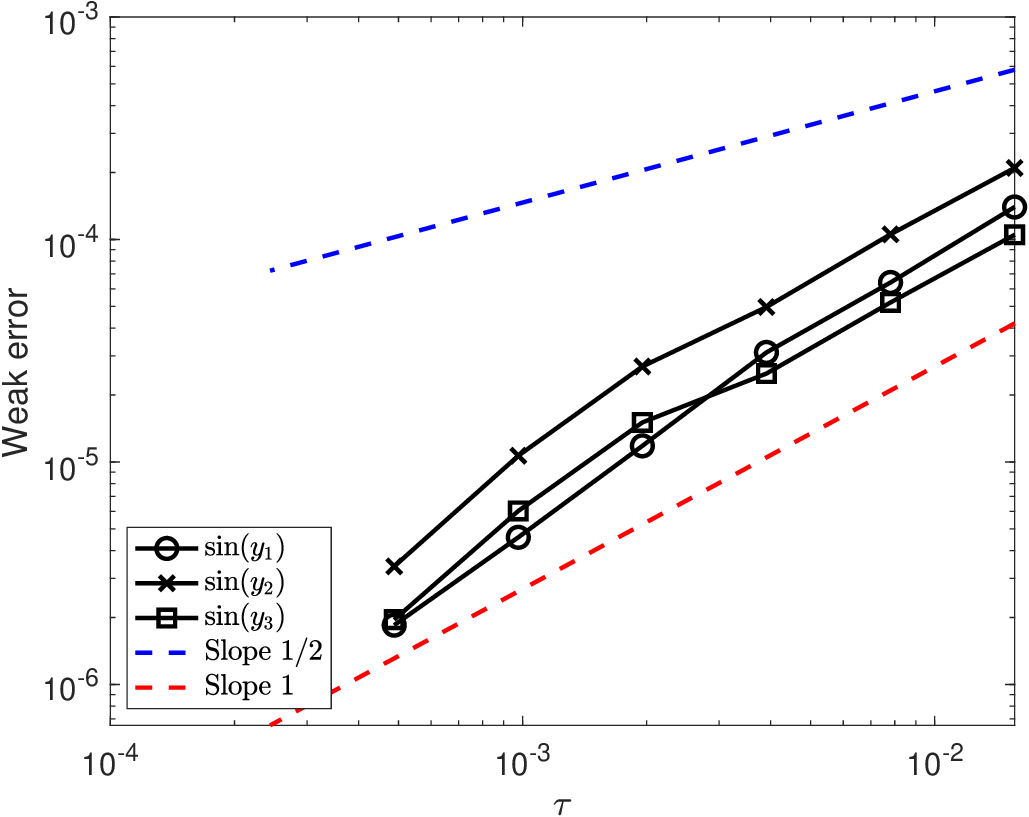}
  \caption{Weak error for $\mathbb E[\sin(y_j)]$ for $j=1,2,3$}
\end{subfigure}%

\caption{Linearly damped stochastic rigid body system~\eqref{eq:srb} (with $3$ noises): Weak convergence of the stochastic conformal exponential integrator.}
\label{fig:wSRB}
\end{figure}

We finally illustrate the property of energy balance of the stochastic conformal exponential integrator given in Proposition~\ref{prop:balanceNum}.
The results are presented in Figure~\ref{fig:trajSRBh}. This is done for the linearly damped SDE
\begin{equation}\label{eq:srbH}
\diff y(t)=B(y(t))\nabla H_0(y(t))\left( \diff t+\circ \diff W(t) \right)-\gamma(t)y(t)\diff t
\end{equation}
with $B$ and $H_0$ from the system~\eqref{eq:srb}. We use the same parameters as for in the first numerical experiment
but take the damping function to be $\gamma(t)=\sin(t)$. The correct preservation of the energy balance
by the stochastic conformal exponential integrator~\eqref{confexp} is numerically illustrated in this figure.

\begin{figure}[h]
  \centering
  \includegraphics[width=0.4\textwidth]{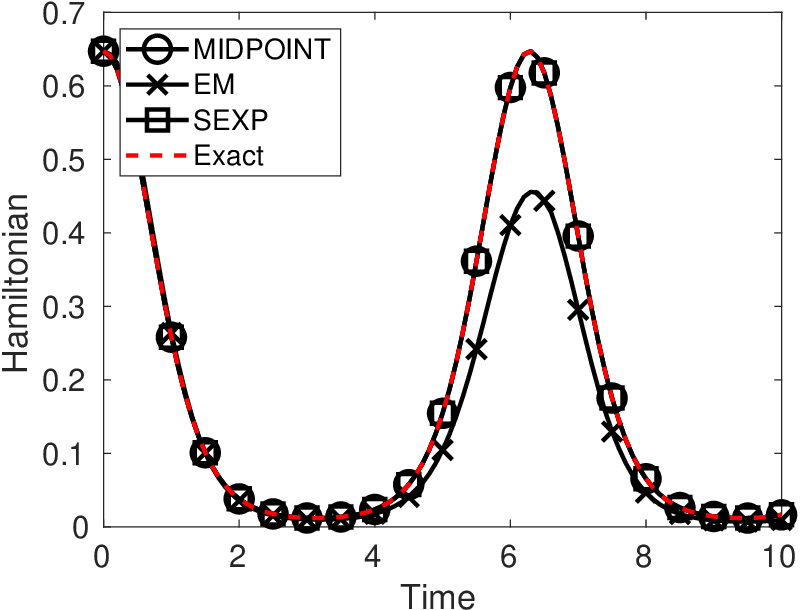}
  \caption{Linearly damped stochastic rigid body system~\eqref{eq:srbH}: Evolution of the quadratic Hamiltonian
  $H_0(y)=\frac12\left( \frac{y_1^2}{I_1}+\frac{y_2^2}{I_2}+\frac{y_3^2}{I_3} \right)$.}
\label{fig:trajSRBh}
\end{figure}

Let us remark that, in all the above numerical experiments, the results of the stochastic conformal exponential integrator~\eqref{confexp}
and of the stochastic midpoint scheme are very similar. This is due to the fact that the Hamiltonians in the linearly damped stochastic rigid body systems are quadratic
(hence the result of the discrete gradient is a midpoint rule) and to the fact that the terms $e^{X^0_{n}}$ and $e^{X^1_{n}}$ are close to the identity. These two time integrators are however not the same, especially when considering non-quadratic Hamiltonians as in the next numerical
experiment.

\subsection{Linearly damped stochastic Lotka--Volterra system}
We consider the following linearly damped stochastic Lotka--Volterra system from Example~\ref{explotka}:
\begin{align}\label{eq:slv}
\diff\begin{pmatrix}y_1\\y_2\\y_3\end{pmatrix}=
B(y)\left(\nabla H(y)\diff t+\nabla H(y)\circ\diff W(t) \right)-\gamma(t)y(t)\diff t,
\end{align}
where $y=(y_1,y_2,y_3)$, the skew-symmetric matrix
$$
B(y)=y_1^{1-ab}y_2^{b+1}\begin{pmatrix}0 & ac & c\\ -ac & 0 & -d\\-c & d & 0\end{pmatrix},
$$
and the Hamiltonian function reads
$$
H(y)=y_1^{ab}y_2^{-b}y_3.
$$
Here, $a,b,c$ are real numbers and $d=abc$. Deterministic Lotka--Volterra systems are considered in, e.g., \cite{MR1044958,MR2784654,MR4242161},
stochastic versions in, e.g., \cite{MR3210739,MR4432606}.

We take the parameters $a=-1,b=-2,c=1$, $\gamma(t)=\sin(t)$, and
the initial value $y(0)=(0.2,0.4,0.6)$. The correct preservation of the energy balance by the stochastic conformal exponential integrator~\eqref{confexp}
is numerically illustrated in Figure~\ref{fig:trajLVh}. This numerically confirms the results of Proposition~\ref{prop:balanceNum} for a non-quadratic Hamiltonian.

\begin{figure}[h]
  \centering
  \includegraphics[width=0.4\textwidth]{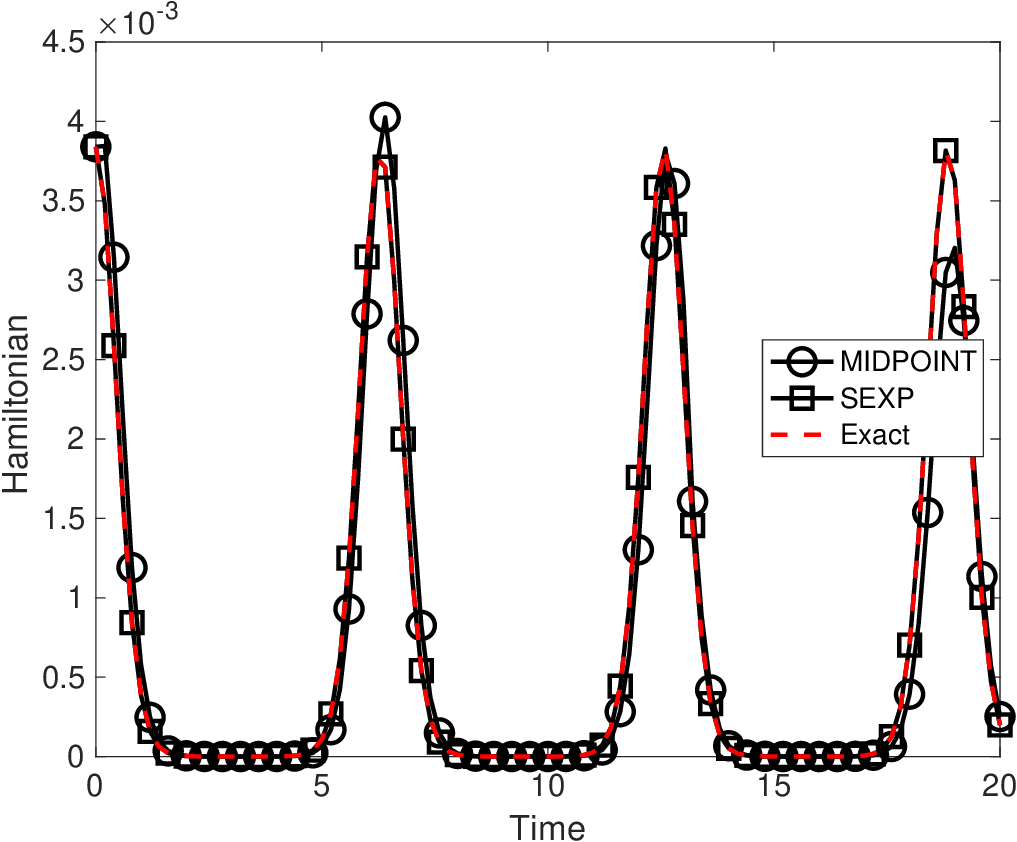}
  \caption{Linearly damped stochastic Lotka--Volterra system~\eqref{eq:slv}: Evolution of the Hamiltonian
  $H(y)=y_1^{ab}y_2^{-b}y_3$.}
\label{fig:trajLVh}
\end{figure}

The strong rate of convergence $1$ of the stochastic conformal exponential integrator can be observed in Figure~\ref{fig:LV}.
To produce this figure, we use $a=-1,b=-1,c=1$ and the same parameters as above. The numerical schemes are applied with the range of
time-step sizes $\dt=2^{-5},\ldots,2^{-12}$. The reference solutions are computed using the conformal exponential integrator
with $\dt_{\text{ref}}=2^{-12}$. The expectations are approximated using $M_s=500$ independent Monte Carlo samples. Observe that the theoretical results from the
previous section cannot be applied here since ${\bf m}(H)=0$.

\begin{figure}[h]
\centering
\includegraphics[width=.4\textwidth]{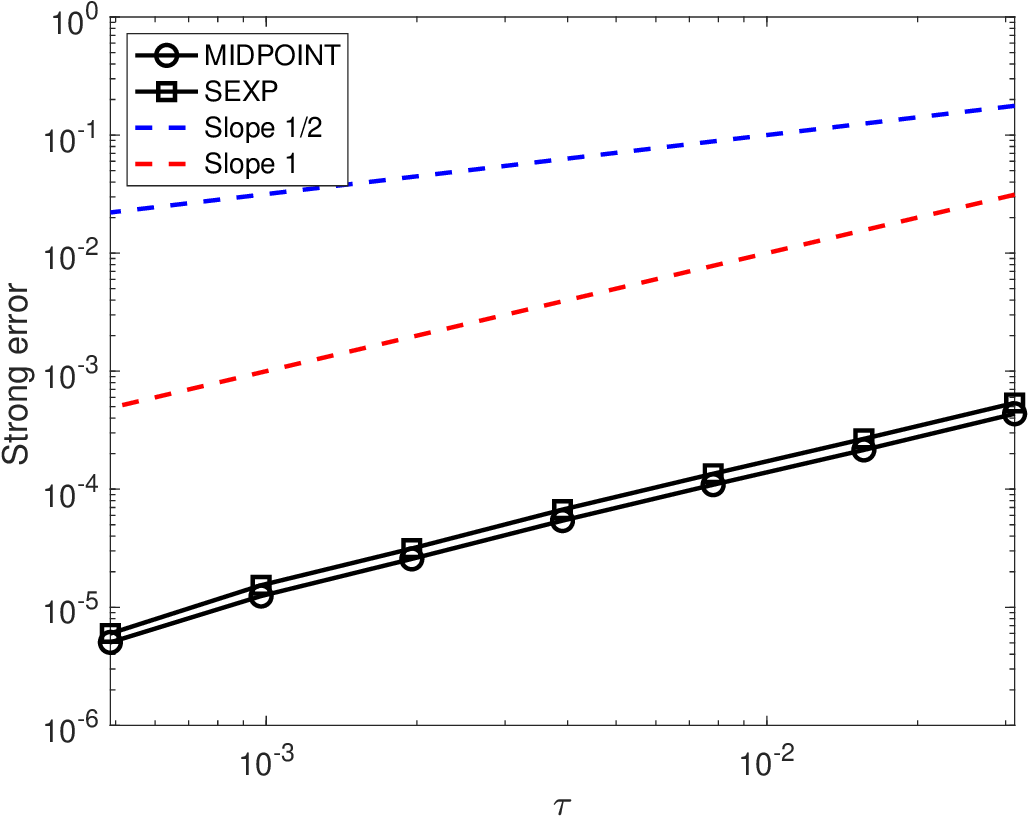}
\caption{Linearly damped stochastic Lotka--Volterra system~\eqref{eq:slv}: Strong convergence of the numerical schemes.}
\label{fig:LV}
\end{figure}

%

\subsection{A linearly damped version of the stochastic Maxwell--Bloch system}
We consider a damped version of the stochastic Maxwell--Bloch system from laser-matter dynamics, see~\cite{MR4593213} and Example~\ref{expmax}:
\begin{align}\label{eq:smb}
\diff\begin{pmatrix}y_1\\y_2\\y_3\end{pmatrix}=
B(y)\left(\nabla H_0(y)\diff t+\nabla H_1(y)\circ\diff W(t) \right)-\gamma(t)y(t)\diff t,
\end{align}
where $y=(y_1,y_2,y_3)$, the skew-symmetric matrix
$$
B(y)=\begin{pmatrix}0 & -y_3 & y_2\\ y_3 & 0 & 0\\-y_2 & 0 & 0\end{pmatrix},
$$
and the Hamiltonian functions read
$$
H_0(y)=\frac12y_1^{2}+y_3\:\text{and}\:H_1(y)=y_3.
$$
This system has the conformal quadratic Casimir $C(y)=\frac12\left(y_2^{2}+y_3^{2}\right)$.
Figure~\ref{fig:trajSMaxBloc} illustrates the preservation of the conformal Casimir quadratic by the stochastic exponential integrator~\eqref{confexp}
as stated in Proposition~\ref{prop:confCasimir}. We use the following parameters: $y(0)=(1,2,3)$, the final time $T=5$, the time-step size $\dt=1/100$
and the damping coefficient $\gamma(t)=\cos(2t)/2$.

\begin{figure}[h]
  \centering
  \includegraphics[width=0.4\textwidth]{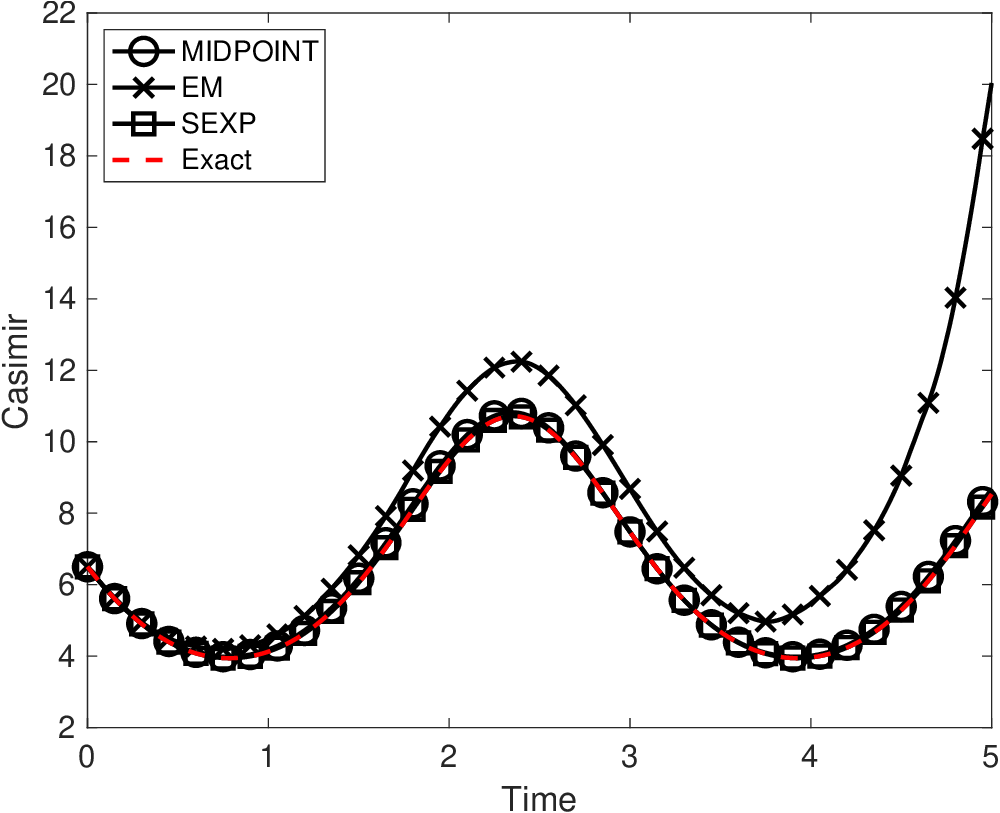}
  \caption{Linearly damped stochastic Maxwell--Bloch system~\eqref{eq:smb}: Evolution of the quadratic Casimir
  $C(y)=\frac12\left(y_2^{2}+y_3^{2}\right)$.}
\label{fig:trajSMaxBloc}
\end{figure}

\subsection{A damped version of the stochastic Poisson system in \cite{MR4396382}}
We consider a damped version of the stochastic Poisson system in \cite{MR4396382}, see Example~\ref{expcao}:
\begin{align}\label{eq:Cao}
\diff\begin{pmatrix}y_1\\y_2\\y_3\end{pmatrix}=
B(y)
\left(\nabla H(y)\diff t+c\nabla H(y)\circ\diff W(t) \right)-\gamma(t)y\diff t,
\end{align}
where $y=(y_1,y_2,y_3)$, the skew-symmetric constant matrix
$$
B(y)=\begin{pmatrix}0 & 1 & -1\\ -1 & 0 & 1\\1 & -1 & 0\end{pmatrix},
$$
and the Hamiltonian function
$$
H(y)=\sin(y_1)+\sin(y_2)+\sin(y_3).
$$
Let us first observe that the SDE~\eqref{eq:Cao} has the quadratic Casimir $C(y)=\frac12y^TDy$ with
the matrix
$$
D=\begin{pmatrix}1 & 1 & 1\\ 1 & 1 & 1\\1 & 1 & 1\end{pmatrix}.
$$
We are thus in the setting of Proposition~\ref{prop:Casimir}, resp. Proposition~\eqref{prop:confCasimir}.
Let us apply the stochastic conformal exponential integrator~\eqref{confexp} with the following parameters:
$T=2$, $y_0=(1,2,3)$, $\gamma(t)=t$, $c=1$, and $\dt=0.02$. The evolution of the quadratic Casimir
along the numerical solution of the proposed integrator can be seen in Figure~\ref{fig:trajCao}. The result is in agreement
with Proposition~\ref{prop:confCasimir}. Note also the slightly less favorable behaviour of the classical Euler--Maruyama scheme
and the good performance of the stochastic midpoint scheme, see below for a discussion on this performance.

\begin{figure}[h]
  \centering
  \includegraphics[width=0.4\textwidth]{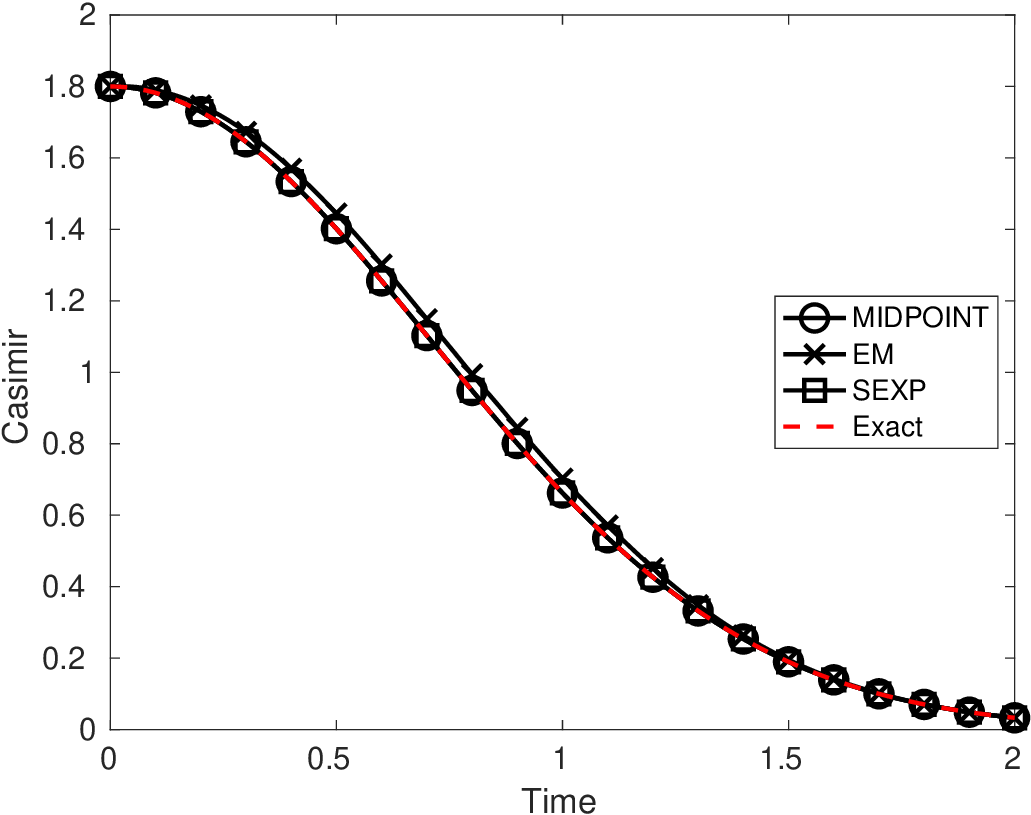}
  \caption{Linearly damped stochastic Poisson system~\eqref{eq:Cao}: Evolution of the quadratic Casimir $C(y)=\frac12y^TDy$.}
\label{fig:trajCao}
\end{figure}

The goal of the next numerical experiment is to illustrate the strong rate of convergence of the stochastic conformal exponential integrator
as stated in Theorem~\ref{th:strong1}. Note that the SDE~\eqref{eq:Cao} has globally Lipschitz coefficients.
The strong rate of convergence is illustrated in Figure~\ref{fig:Cao}.
We use the following parameters: $T=0.5$, $y_0=(1,2,3)$, $c=1$, and the damping function $\gamma(t)=\frac12\cos(2t)$.
The numerical schemes are applied with the time-step sizes $\dt=2^{-5},\ldots,2^{-12}$.
The reference solution is computed by stochastic conformal exponential integrator~\eqref{confexp}
with $\dt_{\text{ref}}=2^{-14}$. The expectations are approximated using $M_s=400$ independent Monte Carlo samples. We have verified that this is enough for the Monte Carlo error
to be negligible. A strong order of convergence $1$ is observed for the proposed exponential integrator, the same order as the stochastic midpoint scheme
(see for instance \cite[Theorem~2.6]{MR1951908}). The strong order of convergence of the Euler--Maruyama scheme is observed to be $1/2$. This is what is expected
in this standard setting of globally Lipschitz coefficients, see for example \cite{MR1214374}. We do not display plots for the weak errors since, in the present setting,
it is clear that the weak rates of convergence are $1$ for these numerical methods.

\begin{figure}[h]
\centering
\includegraphics[width=0.4\textwidth]{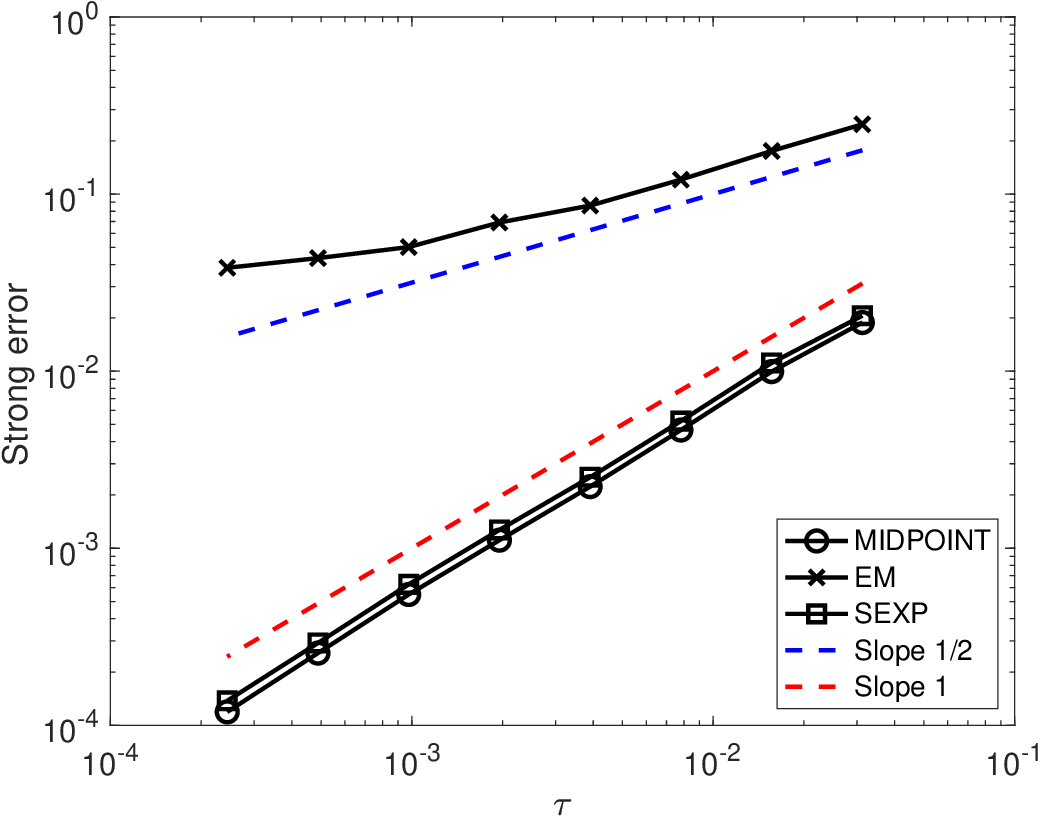}
\caption{Linearly damped stochastic Poisson system~\eqref{eq:Cao}: Strong convergence.}
\label{fig:Cao}
\end{figure}

%

\bibliographystyle{plain}
\bibliography{labib}

\section{Acknowledgements}
The work of CEB is partially supported by the SIMALIN (ANR-19-CE40-0016) project operated by the French National Research Agency.
The work of DC was partially supported by the Swedish Research Council (VR) (projects nr. $2018-04443$
and $2024-04536$) and partially supported by the European Union (ERC, StochMan, 101088589, PI A. Lang).
Views and opinions expressed are however those of the author(s) only and do not necessarily reﬂect those
of the European Union or the European Research Council. Neither the European Union nor the granting authority can be held responsible for them.
The work of YK was partially supported by JSPS Grant-in-Aid for Scientific Research 22K03416.
The computations were performed on resources provided by
the National Academic Infrastructure for Supercomputing in Sweden (NAISS) at Dardel, KTH, and at Vera, Chalmers e-Commons
at Chalmers University of Technology and partially funded by the Swedish Research Council
through grant agreement no. 2022-06725.

\end{document}